\documentclass[11pt]{amsart}
\usepackage{tikz-cd}
\usepackage{amsfonts}
\usepackage{mathrsfs}
\usepackage{amsxtra,amsthm,amsmath,amssymb}
\theoremstyle{plain}

\newcommand{\cleqn}{\setcounter{equation}{0}}
\newcommand{\clth}{\setcounter{theorem}{0}}
\newcommand {\sectionnew}[1]{\section{#1}\cleqn\clth}

\newtheorem{theorem}{Theorem}[section]
\newtheorem{lemma}[theorem]{Lemma}
\newtheorem{definition-theorem}[theorem]{Definition-Theorem}
\newtheorem{proposition}[theorem]{Proposition}
\newtheorem{corollary}[theorem]{Corollary}
\newtheorem{definition}[theorem]{Definition}
\newtheorem{example}[theorem]{Example}
\newtheorem{remark}[theorem]{Remark}
\newtheorem{notation}[theorem]{Notation}
\newtheorem{assumption}[theorem]{Assumption}
\newtheorem{lemma-definition}[theorem]{Lemma-Definition}
\newtheorem{lemma-notation}[theorem]{Lemma-Notation}
\newtheorem{question}[theorem]{Question}
\newtheorem{remark-definition}[theorem]{Remark-Definition}
\newcommand \bthm[1] { \begin{theorem}\label{t#1} }
\newcommand \ble[1] { \begin{lemma}\label{l#1} }

\newcommand \bpr[1] { \begin{proposition}\label{p#1} }
\newcommand \bco[1] { \begin{corollary}\label{c#1} }
\newcommand \bde[1] { \begin{definition}\label{d#1}\rm }
\newcommand \bex[1] { \begin{example}\label{e#1}\rm }
\newcommand \bre[1] { \begin{remark}\label{r#1}\rm }

\newcommand \bnota[1] {\begin{notation}\label{n#1}\rm }
\newcommand \bas[1] { \begin{assumption}\label{a#1}\rm }

\newcommand \bqu[1] { \begin{question}\label{q#1}\rm }

\newcommand {\ethm} { \end{theorem} }
\newcommand {\ele} { \end{lemma} }

\newcommand {\epr} { \end{proposition} }
\newcommand {\eco} { \end{corollary} }
\newcommand {\ede} { \end{definition} }
\newcommand {\eex} { \end{example} }
\newcommand {\ere} { \end{remark} }
\newcommand {\enota} { \end{notation} }
\newcommand {\eas} {\end{assumption}}

\newcommand {\equ} {\end{question}}

\newcommand \lb[1]{\label{#1}}

\def \fg  {\mathfrak{g}}   
\def \fh  {\mathfrak{h}}

\def \fb  {\mathfrak{b}}

\def \fu  {\mathfrak{u}}

\def \ft  {\mathfrak{t}}

\def \fX {\mathfrak{X}}

\def \sB {{\scriptscriptstyle B}}
\def \sC {{\scriptscriptstyle C}}

\def \sG {{\scriptscriptstyle G}}

\def \sP {{\scriptscriptstyle P}}
\def \sQ {{\scriptscriptstyle Q}}
\def \sU {{\scriptscriptstyle U}}

\def \rank { {\mathrm{rank}} }

\newcommand{\beqa}{\begin{eqnarray*}}                     
\newcommand{\eeqa}{\end{eqnarray*}}
\def \hs {\hspace{.2in}}
\def \lara {\la \, , \, \ra}

\def \bfu {{\bf u}}

\def \calO{\mathcal{O}}

\def \pist {\pi_{\rm st}}
\def \Pist {\Pi_{\rm st}}

\def \wF_mn {\wF_m \times \wF_n}
\def \wF_mnC {\wF_{m, n, \, \sC}}

\def \wF {\widetilde{F}}

\def \bfu {{\bf u}}

\def \pist {\pi_{\rm st}}
\def \sP {{\scriptscriptstyle P}}
\def \sQ {{\scriptscriptstyle Q}}

\def \sG {{\scriptscriptstyle G}}
\def \sT {{\scriptscriptstyle T}}

\def \sC {{\scriptscriptstyle C}}
\def \sB {{\scriptscriptstyle B}}

\def \TT {\mathbb{T}}

\def \CC {{\mathbb C}}

\def \Pist {\Pi_{\rm st}}

\def \calT {{\mathcal{T}}}

\def \pist {\pi_{\rm st}}
\def \Pist {\Pi_{\rm st}}

\def \Rst {R_{\rm st}}

\def \lara {\langle \, , \, \rangle}

\begin{document}

\setlength{\baselineskip}{1.2\baselineskip}
\title[]{On the standard Poisson structure and a Frobenius splitting of the basic affine space}
\author{Jun Peng}
\address{
Department of Mathematics   \\
The University of Hong Kong \\
Pokfulam Road               \\
Hong Kong}
\email{genesis@connect.hku.hk}
\author{Shizhuo Yu}
\address{
Department of Mathematics   \\
The University of Hong Kong \\
Pokfulam Road               \\
Hong Kong}
\email{yusz@connect.hku.hk}
\date{}
\begin{abstract}
The goal of this paper is to construct a Frobenius splitting on $G/U$ via the Poisson geometry of $(G/U,\pi_{\sG/\sU})$, where $G$ is a simply-connected semi-simple algebraic group defined over an algebraically closed field of characteristic $p > 3$, $U$ is the uniradical of a Borel subgroup of $G$ and $\pi_{\sG/\sU}$ is the standard Poisson structure on $G/U$. We first study the Poisson geometry of $(G/U,\pi_{\sG/\sU})$. Then, we develop a general theory for Frobenius splittings on $\mathbb{T}$-Poisson varieties, where $\mathbb{T}$ is an algebraic torus. In particular, we prove that compatibly split subvarieties of Frobenius splittings constructed in this way must be $\mathbb{T}$-Poisson subvarieties. Lastly, we apply our general theory to construct a Frobenius splitting on $G/U$.
\end{abstract}
\maketitle

\sectionnew{Introduction and statements of results}\lb{intro}

\subsection{Introduction}\lb{subsec-intro}
Let $G$ be a simply-connected semi-simple algebraic group over an algebraically closed field $k$ and $U$ the uniradical of a Borel subgroup of $G$. The quasi-affine variety $G/U$ is called \emph{the basic affine space}.  When $k=\CC$, there exists the well-known standard Poisson structure $\pist$ on $G$, which has its roots in the theory of quantum groups (see \cite{Q1} \cite{Q2}).  The standard Poisson structure $\pist$ projects to a well-defined Poisson structure $\pi_{\sG/\sU}$ called the standard Poisson structure on $G/U$.

Let $\mathbb{T}$ be a complex torus. A $\mathbb{T}$-Poisson manifold, defined in \cite[Section 1.1]{Tleaves},  is a complex Poisson manifold $(X,\pi_X)$ with a $\mathbb{T}$-action preserving the Poisson structure
$\pi_X$. A $\mathbb{T}$-Poisson manifold gives rise to a decomposition of $X$ into the union of $\mathbb{T}$-leaves  of $\pi_X$, a term defined
in \cite[Section 2.2]{Tleaves}. Every single $\mathbb{T}$-leaf in $(X,\pi_X)$ admits a non-zero anti-canonical section on $X$ called the Poisson $\mathbb{T}$-Pfaffian, a term suggested by A. Knutson and M. Yakimov (see \cite[Section 1.1]{Tleaves}).
When $k=\CC$, $G$ becomes a complex semi-simple Lie group. Let $T$ be a maximal torus of $G$ normalizing $U$.  The left translation of $T$ on $G/U$ makes  $(G/U,\pi_{\sG/\sU})$ into a natural $T$-Poisson manifold. Although J.-H. Lu and V. Mouquin develop a general theory (see \cite[Section 1.5]{Tleaves}) to compute the $\mathbb{T}$-leaves of a class of $\mathbb{T}$-Poisson manifolds, the case $(G/U,\pi_{\sG/\sU})$ is an exceptional example. Using the Bruhat decomposition on $G/U$ and some propositions about the so-called $\mathbb{T}$-mixed Poisson structures (see \S\ref{T-mixed}), we figure out the $T$-leaf decomposition of $(G/U,\pi_{\sG/\sU})$.

Let $X$ be a scheme defined over an algebraically closed field $k$ of positive characteristics. In 1985, Mehta and Ramanathan introduced in \cite{Me-Ra} the notion of Frobenius splittings on $X$ as a $\mathcal{O}_X$-linear map
$
\varphi: F_*\mathcal{O}_X\rightarrow \mathcal{O}_X
$
that splits $F^\sharp: \mathcal{O}_X \rightarrow F_*\mathcal{O}_X $, where $F$ is the Frobenius endomorphism of $X$. They proved that both  Bott-Samelson varieties and Schubert varieties are Frobenius split and obtained a cohomology vanishing
result for line bundles on both Bott-Samelson varieties and Schubert varieties as a consequence. See \cite{linebundle,Me-Ra} for references.
The notion of Poisson $\mathbb{T}$-Pfaffian can be also defined on nonsingular $\mathbb{T}$-Poisson varieties defined over an algebraically closed field of positive characteristic (see Definition \ref{tpoissondef}). Recall that a Frobenius near-splitting on $X$ is an $\mathcal{O}_X$-linear map $\varphi: F_*\mathcal{O}_X\rightarrow \mathcal{O}_X$. Given a Poisson $\mathbb{T}$-Pfaffian $\sigma$ on a $\mathbb{T}$-Poisson variety $(X,\pi_X)$, one can then use $\sigma^{p-1}$ to define a Frobenius near-splitting, where $p={\rm char}k$. Furthermore, we prove that when the near-splitting is a splitting, then all compatibly split subvarieties of the splitting must be $\mathbb{T}$-Poisson subvarieties.

It is well-known that complex simple Lie algebras are classified into types $A,B,C,D,E,F,G$. Over an algebraically closed field $k$ of characteristic $p>0$, recall from \cite[\S 25.4]{LAbook} that a Chevalley algebra can be constructed via reduction modulo $p$ by choosing a Chevalley basis of a complex semi-simple Lie algebra. Then, Chevalley groups are obtained from Chevalley algebras and any simply-connected semi-simple algebraic group $G$ over $k$ is isomorphic to one of the Chevalley groups. See \cite{Ch} and \cite[\S 25.5]{LAbook} for more details. Assume additionally that $p>3$, our general theory can be applied to show that Frobenius splittings can be constructed on Bott-Samelson varieties via Poisson geometry, which share the same property as the canonical splittings constructed in \cite[Theorem 2.2.3]{BK} that all the sub-Bott-Samelson varieties are compatibly-split. Lastly, we apply our general theory to reach our goal, which is to construct a Frobenius splitting on $G/U$ via Poisson geometry.

\subsection{Poisson geometry of $(G/U,\pi_{\sG/\sU})$}\label{Sec12}
When $k=\CC$, let $\fg$ be the Lie algebra of $G$. Let $(B,B_-)$ be a pair of opposite
Borel subgroups of $G$.  Let $T=B\cap B_-$ a maximal torus of $G$ whose Lie algebra is denoted by $\fh$. Let $\lara_\fg$ be a multiple of Killing form on $\fg$.  Let $\pist$ be the standard Poisson structure on $G$ (see $\S$\ref{stPoisson} for detail), which is preserved under the left translation of $T$.  The Poisson structure $\pist$ projects to a $T$-invariant Poisson structure on $G/U$, denoted as $\pi_{\sG/\sU}$. We call $\pi_{\sG/\sU}$ \emph{the standard Poisson structure} on $G/U$.

Let $W=N_{\sG}(T)/T$ be the Weyl group of $(G,T)$, where $N_\sG(T)$ is the normalizer of $T$ in $G$. On $G/U$, there exists the Bruhat decomposition
\[
G/U=\bigsqcup_{w\in W}BwB/U,
\]
such that each $BwB/U$ is a Poisson submanifold of $(G/U,\pi_{\sG/\sU})$.
On each $BwB/U$, there exists an isomorphism
\[
\kappa_{\dot w}: BwB/U \rightarrow T\times (BwB/B)
\]
defined in (\ref{eqkappa}).
On $T\times (BwB/B)$, there exists a $\mathbb{T}$-mixed Poisson structure defined by
\begin{equation*}
0\bowtie_{A_0}\pi_{\sG/\sB}= ((0,\pi_{\sG/\sB})+\sum_k(\rho(H_k),0)\wedge(0,\lambda(H_k)),
\end{equation*}
where $\pi_{\sG/\sB}$ is the projection of $\pist$ to $G/B$, $(H_k)$ is an orthogonal basis of $(\fh,\langle,\rangle_\fg)$, $\rho$ and $\lambda$ are natural Lie algebra actions of $\fh$ on $T$ and $BwB/B$ respectively.

 The following Theorem A is the main tool of this paper in Poisson geometry which is applied in both Theorem B and Theorem E.
\begin{flushleft}
\textbf{Theorem A.} {\it
The isomorphism
\[\kappa_{\dot w}:(BwB/U,\pi_{\sG/\sU})\rightarrow (T\times (BwB/B),0\bowtie_{A_0}\pi_{\sG/\sB})\]
is Poisson.}
\end{flushleft}

Let $(X,\pi_X)$ be a $\mathbb{T}$-Poisson manifold. Suppose that $\Sigma$ is a symplectic leaf in $(X,\pi_X)$. The set
\[
\mathbb{T}\Sigma=\bigsqcup_{t\in \mathbb{T}}t\Sigma
\]
is called \emph{a single $\mathbb{T}$-leaf in $(X,\pi_X)$}. See \cite[Section 2.2]{Tleaves} for more details.

The left translation of $T$ on $G/U$ makes  $(G/U,\pi_{\sG/\sU})$ into a natural $T$-Poisson manifold. Using Theorem A, we figure out the $T$-leaves decomposition of $(G/U,\pi_{\sG/\sU})$.

\begin{flushleft}
\textbf{Theorem B.} {\it The decomposition of $G/U$ into $T$-leaves of the Poisson structure $\pi_{\sG/\sU}$ is
\[
G/U=\bigsqcup_{v\leq u}(BuB/U\cap B_-vB/U).
\]
 }
\end{flushleft}

\subsection{General construction of Frobenius splittings via Poisson geometry}\lb{subsec-nota-intro}
To study Frobenius splittings on $G/U$, where $G$ is defined over an algebraically closed field $k$ of characteristic $p>3$, we first develop a general construction via Poisson geometry. The details about Poisson algebras in positive characteristics are presented in \S\ref{ppc}.

\begin{definition}
{\rm	Let $X$ be a non-singular variety over $k$. A Poisson structure on $X$ is a section $\pi \in H^0(X, \wedge^2 \calT_X)$ satisfying $[\pi,\pi]=0$, where $[-,-]$ is the Schouten bracket (see \S \ref{ppc} for the definition of the Schouten bracket). A (non-singular) Poisson variety is a non-singular variety equipped with a Poisson structure.}
\end{definition}

\begin{definition}\label{tpoissondef}
{\rm	Let $\TT$ be an algebraic torus over $k$. A $\mathbb{T}$-Poisson variety is a Poisson variety $X$ with a $\TT$-action by Poisson automorphisms.
	A Poisson $\TT$-Pfaffian is a section $\sigma \in H^0(X, \omega_X^{-1})$ such that
	\[
	\sigma={\pi}^{[r]} \wedge \partial_{h_1} \wedge \ldots \wedge \partial_{h_l},
	\]
	where $r$ is half of the rank of $\pi$ and $\partial_{h_i}$ is the vector field on $X$ generated by some element $h_i \in \mathfrak{t} = {\rm Lie}(\TT)$. (See Definition \ref{gvf} for the definition of generating vector fields)}
\end{definition}

Let $s \in H^0(X,\omega_X^{1-p})$. One can regard $s$ as an element in $\rm{Hom}_{\calO_X}(\omega_X, F_* \omega_X) $ by $s(\gamma)=(s|_U, \gamma^{p-1}) \gamma$, for $\gamma \in H^0(\mathcal{U}, \omega_X)$, where $\mathcal{U}$ is an open subset of $X$. It is well-known that the Cartier operator gives rise to the isomorphism (\cite[Proposition 1.3.7]{BK})
\[
\mathcal{H}om_{\calO_X}(\omega_X, F_* \omega_X) \cong \mathcal{H}om_{\calO_X}(F_* \calO_X, \calO_X).
\]
Therefore in particular the $p-1$ tensor power $\sigma^{p-1}$ of a Poisson $\mathbb{T}$-Pfaffian $\sigma$ gives rise to a Frobenius near-splitting $\phi_\sigma$ of $X$.

\begin{definition}
{\rm A Frobenius Poisson $\mathbb{T}$-Pfaffian of a $\mathbb{T}$-Poisson variety $(X, \pi_X)$ is a Poisson $\mathbb{T}$-Pfaffian $\sigma$ such that the Frobenius near-splitting constructed in the way above by some non-zero scalar multiple of $\sigma$ is a Frobenius splitting.}
\end{definition}

\begin{flushleft}
	\textbf{Remark.}
	{\rm A basic fact about Frobenius splitting is that if one can verify the restriction of a near-splitting $\varphi$ on a variety $X$ to an open subset $\mathcal{U}$ of $X$ is a Frobenius splitting on $\mathcal{U}$, then $\varphi$ is a Frobenius splitting on $X$. See \cite[Section 1.1]{BK} for more details. The advantage of constructing the near-splitting via Poisson $\mathbb{T}$-Pfaffian is that often there exists some nice expression of the Poisson structure in local coordinates.}
\end{flushleft}

In \S\ref{CGL} we consider a special kind of $\mathbb{T}$-Poisson structure on affine spaces called \emph{Poisson CGL extensions},  a term coined by K.Goodearl and M.Yakimov and named after G.Cauchon, K.Goodearl, E.Letzter (\cite[Section 1.3]{PoissonCGL}). They play an important role in the study of cluster algebras (see \cite{PoissonCGL}) as well as integrable systems (see \cite{Integrablesystems}). We define the notion of a Poisson CGL bivector field in Definition \ref{PCGLV} and its log-canonical part in Remark \ref{lc}. Then, we obtain a sufficient condition such that a Frobenius Poisson $\TT$-Pfaffian exists on a $\TT$-Poisson variety.
\begin{flushleft}
\textbf{Theorem C.}
{\it If the rank of a Poisson CGL bivector field $\pi_X$ is equal to that of $(\pi_X)_0$, the log-canonical part of $\pi_X$, then a Frobenius Poisson $\TT$-Pfaffian exists.}
\end{flushleft}

For a general $\mathbb{T}$-Poisson variety $(X, \pi_X)$, the connection between compatibly split subvarieties of $X$ and $\mathbb{T}$-Poisson subvarieties of $(X,\pi_X)$ is presented in the following Theorem D.

\begin{flushleft}
\textbf{Theorem D.} {\it Let $\sigma$ be a Poisson $\mathbb{T}$-Pfaffian of a $\mathbb{T}$-Poisson variety $(X, \pi_X)$. If the Frobenius near-splitting $\phi_\sigma$ defined by $\sigma$ is a Frobenius splitting, then all compatibly split subvarieties
of $\phi_\sigma$ must also be $\mathbb{T}$-Poisson subvarieties with respect to the Poisson
structure $\pi_X$. }
\end{flushleft}

\subsection{A Frobenius Splitting of $G/U$ via Poisson $T$-Pfaffian}
Theorem C can be applied to concrete examples including Bott-Samelson varieties (see \S \ref{bottdef} for the definition of Bott-Samelson varieties) and the basic affine space when $k$ is of characteristic $p > 3$.

First, we equip the Bott-Samelson varieties with the standard Poisson structure as well as a torus action (see the paragraph below Remark \ref{sdpoissonpcc} for the definition of the standard Poisson structure of Bott-Samelson varieties and basic affine spaces over positive characteristic). Then, Theorem \ref{bottsplit} proves that a Frobenius Poisson $T$-Pfaffian exists on any Bott-Samelson varieties and is unique up to scalar multiple. Therefore, a Frobenius splitting $\phi_\sigma$ defined in \S\ref{subsec-nota-intro} exists on any Bott-Samelson variety.
Moreover, Frobenius splittings on Bott-Samelson varieties can be also obtained from a class of special anti-canonical sections by \cite[Theorem 2.2.3]{BK} and the canonical Frobenius splittings are constructed in the proof of \cite[Theorem 2.2.3]{BK}. On any Bott-Samelson variety, the Frobenius Poisson $T$-Pfaffian constructed in Theorem \ref{bottsplit} provides a choice of
anti-canonical section in \cite[Theorem 2.2.3]{BK} and a natural question is raised as follows.

\begin{question}
{\rm On any Bott-Samelson variety, is the Frobenius splitting induced by the Poisson $T$-Pfaffian in Theorem \ref{bottsplit} same as the Frobenius splitting constructed in the proof of \cite[Theorem 2.2.3]{BK} ?}
\end{question}

Based on Theorem \ref{bottsplit} and Theorem A, we reach our goal of this paper, which is to construct a Frobenius splitting on $G/U$ via the Poisson geometry as follows.
\begin{flushleft}
\textbf{Theorem E.} {\it Let $G$ be a simply-connected semi-simple algebraic group over an algebraically closed field $k$ of characteristic $p > 3$, $U$ the uniradical of a Borel subgroup of $G$. Let $T$ be a maximal torus of $G$ normalizing $U$, which acts on $G/U$ by left translation. Denote by $\pi_{\sG/\sU}$ the standard Poisson structure on $G/U$. Then there exists a unique (up to scalar multiples) Poisson $T$-Pfaffian $\sigma$ on $(G/U,\pi_{\sG/\sU})$ and $\sigma^{p-1}$ induces a Frobenius splitting on $G/U$.
 }
 \end{flushleft}

\subsection{Acknowledgments}
The authors would like to thank Jiang-Hua Lu for her help and encouragements.
This work was completed while the authors were supported by the University of Hong Kong Postgraduate
Studentship.
\sectionnew{Poisson geometry of $(G/U,\pi_{\sG/\sU})$}
\subsection{Basics on Poisson geometry}
Let $G$ be a complex Lie group and let $\pi_\sG$ be a holomorphic Poisson structure on $G$. The pair
$(G, \pi_\sG)$ is called \emph{a Poisson Lie group} if the multiplication map
\[(G\times G,\pi_\sG\times\pi_\sG)\rightarrow (G,\pi_\sG)\]
is Poisson. See \cite{Q2,lu-we:poisson} for references.
It is well known (see \cite[Theorem 1.2]{lu-we:poisson}) that $(G,\pi_\sG)$ is a Poisson Lie group if and only if $\pi_\sG$ is \emph{multiplicative}, i.e.
\begin{equation}
\pi_\sG(gh)=l_g\pi_\sG(h)+r_h\pi_\sG(g),\hs g,h\in G.
\end{equation}
\begin{lemma}\cite[Theorem 1.2]{lu-we:poisson}
Let $(G,\pi_\sG)$ be a Poisson Lie group. Then, $\pi_\sG$ vanishes at the identity and the group inversion map is anti-Poisson.
\end{lemma}

Let $(G,\pi_\sG)$ be a complex Poisson Lie group and let $(P,\pi_\sP)$ be a complex Poisson variety such that $G$ acts on $P$ from the left preserving the Poisson structure. If the action is transitive, $(P,\pi_P)$ is called a
\emph{Poisson homogeneous space} of $(G,\pi_\sG)$. See \cite{Poissonembedding} for reference.

A subgroup/variety $Q$ of a Poisson Lie group/variety $(G,\pi_\sG)$ is called \emph{a coisotropic subgroup/variety} if for any $q\in Q$
\[\pi_\sG(q)\in T_qQ\wedge T_qG\subset \wedge^2 T_qG.\]
The following proposition proved in \cite[Theorem 6]{STS} provides a class of examples of Poisson homogeneous spaces.
\begin{proposition}\label{qp}
Let $(G,\pi_\sG)$ be a Poisson Lie group and let $Q\subset G$ be a closed coisotropic Lie subgroup of $(G,\pi_{\sG})$. Denote the projection from $G$ to $G/Q$ by $p_{\sG/\sQ}$.
Then $p_{\sG/\sQ}(\pi_\sG)$ is a well defined Poisson structure on $G/Q$ and $(G/Q,p_{\sG/\sQ}(\pi_\sG))$
is a Poisson homogeneous space of $(G,\pi_\sG)$.
\end{proposition}
Let $\mathbb{T}$ be a complex torus. Recall that a complex $\mathbb{T}$-Poisson variety is a Poisson variety $(P,\pi_\sP)$ with a $\TT$-action by Poisson automorphisms. Let $\mathbb{T}_1$ and $\mathbb{T}_2$ be complex tori with Lie algebras $\ft_1$ and $\ft_2$ respectively.  Let $(P,\pi_\sP)$ be a complex $\mathbb{T}_1$-Poisson variety and $(Q,\pi_\sQ)$ be a complex $\mathbb{T}_2$-Poisson variety. The $\mathbb{T}_1$ and $\mathbb{T}_2$ actions are denoted by $\rho$ and $\lambda$ respectively. Given
\[A=\sum_k a_k\otimes b_k\in \ft_1 \otimes\ft_2,\]
define the bivector field on $P\times Q$ by
\begin{equation}\label{TMix}
\pi_\sP\bowtie_A\pi_\sQ= (\pi_\sP,0)+(0,\pi_\sQ)+\sum_k(\rho(a_k),0)\wedge(0,\lambda(b_k)).
\end{equation}
Consider the $(\mathbb{T}_1\times \mathbb{T}_2)$ action on $P\times Q$ defined by
\[(t_1,t_2)(p,q)=(t_1p,t_2q).
\]
The following lemma proved in \cite[Lemma 2.1.1]{Yuthesis}.
\begin{lemma}\label{mixedpoisson1}
For any $A\in\ft_1\otimes\ft_2$, $(P\times Q, \pi_\sP\bowtie_A\pi_\sQ)$ is a $(\mathbb{T}_1\times \mathbb{T}_2)$-Poisson variety.
\end{lemma}
The mixed product Poisson structure on $P\times Q$ associated to the torus actions constructed in Lemma \ref{mixedpoisson1} on $P\times Q$
is simply called \emph{a $\mathbb{T}$-mixed Poisson structure}.
\subsection{The standard Poisson structure $\pi_{\sG/\sU}$ on $G/U$}\label{stPoisson}
Let $G$ be a simply-connected complex semi-simple Lie group and let $\fg$ be the Lie algebra of $G$. Let $(B,B_-)$ be a pair of opposite Borel subgroups of $G$. Denote $U$ and $U_-$ respectively the uniradical of $B$ and $B_-$. Let $T=B\cap B_-$ be a maximal torus of $G$. Denote Lie algebras of $B$, $B_-$, $U$, $U_-$ and $T$ by $\fb$, $\fb_-$, $\fu$, $\fu_-$ and $\fh$ respectively. Let $\Delta \subset \fh^*$ be
the set of roots of $\fg$ with respect to $\fh$, and let
\[\fg = \fh + \sum_{\alpha \in \Delta} \fg_\alpha\]
be the root decomposition of $\fg$. Let $\Delta_+ \subset \Delta$ and $\Gamma \subset \Delta_+$
be respectively the sets of positive roots and simple roots. We also write $\alpha > 0$ for $\alpha \in \Delta_+$.

For $\alpha > 0$, let $h_\alpha$ be the unique element in $[ \mathfrak{g}_\alpha ,\mathfrak{g}_{-\alpha} ]$ such that $\alpha(h_\alpha)=2$. Fix,
for each $\alpha \in \Delta_+$,  $e_\alpha \in \fg_\alpha$ and $e_{-\alpha} \in \fg_{-\alpha}$
such that $[e_\alpha, e_{-\alpha}]=h_\alpha$. Let $\langle\ , \ \rangle_\fg$ be a fixed multiple of Killing form on $\fg$ such that $\frac{\langle \alpha, \alpha \rangle}{2} \in \{1,2,3\}$ for each root $\alpha$, where $\langle\ ,\ \rangle$ is the induced bilinear form on $\fh^*$.
\begin{remark}
{\rm The existence of such choice follows from the \emph{proof} of \cite[Lemma 5]{charp}.}
\end{remark}

 For $\alpha>0$, denote $E_\alpha=\sqrt{\langle \alpha, \alpha\rangle}e_\alpha$ and $E_{-\alpha}=\frac{\sqrt{\langle \alpha, \alpha\rangle}}{2}e_{-\alpha}$.
Let $(H_i)$ be an orthogonal basis of $(\fh, \langle\ , \ \rangle_\fg)$.
\begin{definition}\cite[Chapter 4]{Q2}\cite[Section 6.1]{Tleaves}\label{sdpoisson}
	{\rm The standard Poisson structure $\pist$ on $G$ is defined by
		\begin{equation}
		\pist(g)= l_g\Lambda_{\rm st} - r_g\Lambda_{\rm st},\ g\in G
		\end{equation}
		where $l_g$ and $r_g$ are respectively the left and right translation of $g$ on $G$ and
		\[
		\Lambda_{\rm st} = \sum_{\alpha > 0} \frac{\langle \alpha, \alpha\rangle}{2} e_{-\alpha} \wedge e_\alpha=\sum_{\alpha > 0}E_{-\alpha}\wedge E_\alpha.
		\]}
\end{definition}

\begin{remark}
{\rm	The Poisson structure $\pist$ depends on the choices of the triple $(B, B_-, \langle\ , \ \rangle_\fg)$, but not on the choices of the root vectors $e_{\pm \alpha} \in \fg_{\pm \alpha}$.}
\end{remark}

Recall that in characteristic $p > 3$, the Lie algebra of a simply-connected semi-simple algebraic group can be constructed by reduction modulo $p$ via a Chevalley basis of a complex semi-simple Lie algebra. Since we have chosen the bilinear form so that $\frac{\langle \alpha, \alpha \rangle}{2} \in \{1,2,3\}$ for each root $\alpha$ in the complex case, $\frac{\langle \alpha, \alpha \rangle}{2}$ is non-zero modulo $p$ and the element
\[
\Lambda_{\rm st} = \sum_{\alpha > 0} \frac{\langle \alpha, \alpha\rangle}{2} e_{-\alpha} \wedge e_\alpha
\]
modulo $p$ is well-defined.

\begin{definition}\label{sdpoissonpc}
{\rm	The standard Poisson structure on a simply-connected semi-simple algebraic group over an algebraically closed field of characteristic $p > 3$ is defined by the same formula as in Definition \ref{sdpoisson}.}
\end{definition}

For a root $\beta$, denote by $u_\beta (z)$ the root-subgroup corresponding to $\beta$, where $z \in  \CC$.
For each $\alpha \in \Gamma$, the simple reflection $s_\alpha\in W$ is represented by
\[
\dot{s}_\alpha=u_{\alpha}(-1)u_{-\alpha}\left(1\right)u_{\alpha}\left(-1\right)
\] in $N_\sG(T)$. We denote by $P_{\alpha}=B \cup B\dot{s}_{\alpha}B$ the
parabolic subgroup corresponding to $\alpha \in \Gamma$. It is well-known that $P_\alpha$ is a Poisson-Lie subgroup of $G$ for all $\alpha$. We denote the restriction of $\pist$ on $P_\alpha$ by the same notation.
Let $(H_i)$ be an orthogonal basis of $(\fh,\langle, \rangle_\fg)$. Consider now the direct product Lie algebra $\fg\oplus \fg$ and define $R_{\rm st}\in (\fg\oplus\fg)^{\otimes 2}$ by
\begin{multline}
\Rst=\sum_{\alpha\in \Delta_+}(E_{-\alpha},E_{-\alpha})\wedge(E_\alpha,0)+\sum_{\alpha\in \Delta_+}(E_{\alpha},E_{\alpha})\wedge(0,E_{-\alpha})\\
+\frac{1}{2}\sum_i(H_i,H_i)\wedge (H_i,-H_i),
\end{multline}
and
\begin{equation}
\Pist=\Rst^l-\Rst^r.
\end{equation}
Then, $(G\times G,\Pist)$ is a Poisson Lie group, which is called the Drinfeld double of $(G,\pist)$. See \cite{DD} for more details.
Recall from \cite[Proposition 2.2]{DD} that
\[G_\Delta=\{(g,g):g\in G\}\subset G\times G\]
 is a Poisson Lie subgroup of $(G\times G,\Pist)$ satisfying the following Lemma \ref{PE}.
\begin{lemma}\label{PE}
The diagonal embedding
\[(G,\pist)\longrightarrow (G\times G,\Pist),\hs g\longmapsto (g,g)\]
is Poisson.
\end{lemma}

It follows from the construction of $\pist$ that $U$ is a coisotropic subgroup of $(G,\pist)$. Moreover, the following lemma is proved in \cite[Section 6.1]{Tleaves}.
\begin{lemma}\label{submanifold}
The standard Poisson structure $\pist$ vanishes on $T$ and $BwB$ is a Poisson subvariety of $(G,\pist)$ for any $w\in W$.
\end{lemma}
\begin{definition}
{\rm The projection of $\pist$ to $G/U$ is a well-defined Poisson structure by Proposition \ref{qp}, which is denoted by $\pi_{\sG/\sU}$ and called the standard Poisson structure on $G/U$.}
\end{definition}
Let $\lambda$ be the $T$-action on $G/U$ defined by
\[
T\times (G/U)\longrightarrow (G/U),\hs (t,\; g.U) \longmapsto tg.U, \hs t \in T, \, g \in G.
\]
\begin{lemma}\label{TPOISSON}
$(G/U,\pi_{\sG/\sU})$ is a $T$-Poisson variety.
\end{lemma}
\begin{proof}
It follows from Proposition \ref{qp} that $(G/U,\pi_{\sG/\sU})$ is a Poisson homogeneous space, i.e. for $t\in T$ and $g\in G$,
\[\pi_{\sG/\sU}(tg.U)=\sigma_{g.U}\pist(t)+\lambda_t\pi_{\sG/\sU}(g.U),\]
where
$\sigma_{g.U}: T\rightarrow G/U,\ g.U\mapsto t'g.U,\ \forall\ t'\in T,$
and
$\lambda_{t}: G/U\rightarrow G/U,\ g'.U\mapsto tg'.U,\ \forall\ g'\in G.$
 As $\pist$ vanishes on $T$, the left translation of $T$ on $G/U$ preserves the Poisson structure $\pi_{\sG/\sU}$, i.e.
\[\pi_{\sG/\sU}(tg.U)=\lambda_t\pi_{\sG/\sU}(g.U).\]
\end{proof}
\subsection{Proof of Theorem A}\label{T-mixed}
Let $G$ be a complex semi-simple algebraic group and follow the same notation as in \S\ref{stPoisson}.
It is well-known that $G$ has the Bruhat decomposition
\[
G=\bigsqcup_{w\in W}BwB.
\]
 As $U$ is a closed subgroup of $B$, $G/U$ also has the Bruhat decomposition
\[
G/U=\bigsqcup_{w\in W}BwB/U.
\]
For each $w\in W$, it follows from Lemma \ref{submanifold} that $BwB$ is a Poisson subvariety of $(G,\pist)$ and $BwB/U$ is a Poisson subvariety of
$(G/U,\pi_{\sG/\sU})$. For $w\in W$, define $U_w=U\cap wU_-w^{-1}$. If we fix a representative $\dot w\in N_\sG(T)$, any $g\in BwB$ can be uniquely decomposed as $g=tx\dot wn$ with $t\in T$, $x\in U_w$ and $n\in U$.
For any $w\in W$, consider the isomorphism on $BwB$ defined by
\begin{equation}
J_{\dot w}:\ BwB\rightarrow T\times (BwB/T),\ \ tx\dot{w}n\mapsto (t,x\dot wn.T).
\end{equation}
and the isomorphism on $BwB/U$ defined by
\begin{equation}\label{eqkappa}
\kappa_{\dot w}: BwB/U \rightarrow T\times (BwB/B),\hs tx\dot w.U\mapsto (t,xw.B).
\end{equation}
Recall from Proposition \ref{qp} that $\pist$ also projects to well-defined Poisson structures on $G/T$ as well as $G/B$ , denoted respectively by $\pi_{\sG/\sT}$ and $\pi_{\sG/\sB}$.
Let $\rho$ be the torus action defined by
\begin{equation}\label{rho}
T\times T\longrightarrow T, \;\;\; (t,t') \longmapsto tt'.
\end{equation}
When $Q=B$ or $T$, let $\lambda$ be the torus action defined by
\begin{equation}\label{lambda}
T\times (BwB/Q)\longrightarrow (BwB/Q), \;\;\; (t,g.Q) \longmapsto tg.Q.
\end{equation}
Let $(H_i)$ be an orthogonal basis of $(\fh,\langle,\rangle_\fg)$ and let
\[A_0=\sum_i H_i\otimes H_i.\]
Similar to the proof of Lemma \ref{TPOISSON}, we know that $(G/Q,\pi_{\sG/\sQ})$ is a $T$-Poisson variety.
Then, $0\bowtie_{A_0}\pi_{\sG/\sQ}$ defined in (\ref{TMix}) is a $\mathbb{T}$-mixed Poisson structure on $T\times (BwB/Q)$.

To prove Theorem A, we first prove the following Proposition \ref{mixedG}, which expresses $J_{\dot w}(\pist)$ as a $\mathbb{T}$-mixed Poisson structure on $T\times (BwB/T)$.
\begin{proposition}\label{mixedG}
The isomorphism
\begin{equation}
J_{\dot w}:\ (BwB,\pist)\rightarrow (T\times (BwB/T),0\bowtie_{A_0}\pi_{\sG/\sT})
\end{equation}
is Poisson.
\end{proposition}
\begin{proof}
First, we prove that the map
\[q_{\dot w}: (BwB,\pist)\rightarrow (T,0),\ \ tx\dot{w}n\mapsto t\]
is Poisson.
For any $g\in BwB$ and a fixed $\dot w\in N_\sG(T)$, $g$ can be uniquely written as
\[g=tx\dot{w}n=tn'\dot{w}x'\]
for the same $t\in T$, $x\in U_w$, $x'\in U\cap w^{-1}U_{-}w$ and $n,n'\in U$.
So, we can write
\[q_{\dot w}=j_\sT\circ j_{\dot w},\]
where
\[j_{\dot w}: (BwB,\pist)\rightarrow (B,\pist),\ tn'\dot{w}x'\mapsto tn',\ t\in T,\ n'\in U,\ x'\in U\cap w^{-1}U_-w,\]
and
\[j_\sT:(B,\pist)\rightarrow (T,0),\ tn'\mapsto t,\ t\in T, n'\in U.\]
It is sufficient to prove that $j_{\dot w}$ and $j_\sT$ are Poisson.

By the multiplicativity of $\pist$, one has
\[
\pist(tx\dot{w}n)=\pist(tn'\dot{w}x')=l_{tn'}\pist(\dot w x')+r_{\dot w x'}\pist(tn'),
\]
where $r_g$ denotes the right translation on $G$ and $l_h$ denotes the left translation on $G$.
On one hand, for a fixed $tn'\in B$, all the elements in $\dot wU\cap U_-\dot w$ go to the same point $tn'$ via $j_{\dot w}\circ l_{tn'}$, i.e. for any $\dot w x'\in \dot wU\cap U_-\dot w$
\[j_{\dot w}l_{tn'}(\dot w x')=tn'.\]
As $\dot wU\cap U_-\dot w$ is coisotropic in $(G,\pist)$, (see \cite[Lemma 10]{groupoid})
one has
\[j_{\dot w}l_{tn'}\pist(\dot wx')=0.\]
On the other hand, for a fixed $\dot w\in N_\sG(T)$ and $x'\in U\cap w^{-1}U_-w$, $j_{\dot w}\circ r_{\dot w x'}$ fixes $B$ pointwise, i.e. for any $tn'\in B$,
\[j_{\dot w}\circ r_{\dot w x'}(tn')=tn'.\]
As $B$ is a Poisson Lie subgroup of $(G,\pist)$, one has
\[j_{\dot w}(\pist(tn'\dot{w}x'))=j_{\dot w}r_{\dot w x'}\pist(tn')=\pist(tn').
\]
Moreover, as $U$ is a coisotropic subgroup of $(B,\pist)$ and $\pist$ vanishes on $T$, one has
\[j_{\sT}\pist(tn')=j_{\sT}l_t\pist(n')=0.\]
Therefore, $j_{\dot w}$ and $j_{\sT}$ are Poisson maps.

Then, we study the mixed terms.
Define
\[\Psi_{\dot w}: (BwB)\times G \rightarrow T\times (G/T),\ \ \ (tx\dot wn,g)\mapsto (t,g.T).\]
Recall from Lemma \ref{PE} that $(BwB,\pist)$ is Poisson isomorphic to $((BwB)_\Delta,\Pist)$,
where
 \[(BwB)_\Delta=\{(g,g):g\in BwB\}\subset G_\Delta.\]
 We can identify $J_{\dot w}\pist$ with $\Psi_{\dot w}(\Pist)$.
By the definition of $\Pist$, write
\begin{equation}
\Pist=(\pist,0)+(0,\pist)+\mu_1+\mu_2+\mu_3+\mu_4,
\end{equation}
where,
\[\mu_1=2\sum_{\alpha\in\Delta_+}(E^r_\alpha,0)\wedge(0,E^r_{-\alpha}),\ \ \ \mu_2=-2\sum_{\alpha\in\Delta_+}(E^l_\alpha,0)\wedge(0,E^l_{-\alpha})\]
and
\[\mu_3=\sum^n_{i=1}(H_i^r,0)\wedge(0,H_i^r),\ \ \  \mu_4=-\sum^n_{i=1}(H_i^l,0)\wedge(0,H_i^l).\]
Since the vector fields $E^l_\alpha$, $E^r_\alpha$ vanish when projected to $T$ and $H_i^l$ on $BwB$ vanishes when projected to $BwB/T$,
one has $\Psi_{\dot{w}}(\mu_1)=\Psi_{\dot w}(\mu_2)=\Psi_{\dot w}(\mu_4)=0$.
Additionally, one has
\[\Psi_{\dot w}(\mu_3)
=\sum_i(\rho(H_i),0)\wedge(0,\lambda(H_i)),
\]
where $\rho$ and $\lambda$ are Lie algebra actions induced by $(\ref{rho})$ and $(\ref{lambda})$ respectively.
Therefore, $J_{\dot w}\pist=0\bowtie_{A_0}\pi_{\sG/\sT}$.
\end{proof}

Furthermore, the following theorem expresses $\kappa_{\dot w}(\pi_{\sG/\sU})$ as a $\mathbb{T}$-mixed Poisson structure on $T\times (BwB/B)$.
\begin{theorem}\label{mixedGU}
The isomorphism
\[\kappa_{\dot w}:(BwB/U,\pi_{\sG/\sU})\rightarrow (T\times (BwB/B),0\bowtie_{A_0}\pi_{\sG/\sB})\]
is Poisson.
\end{theorem}
\begin{proof}
Let $p_{\sG/\sU}$ be the natural projection from $G$ to $G/U$. When restricted to $BwB$, one has
\[J_{\dot w} \circ p_{\sG/\sU}=p_{\sG/\sU}\circ J_{\dot w}.\]
Therefore, on each $BwB/U$, one has
\[\kappa_{\dot w}(\pi_{\sG/\sU})=J_{\dot w} \circ p_{\sG/\sU}(\pist)=p_{\sG/\sU}(J_{\dot w}(\pist))=0\bowtie_{A_0}\pi_{\sG/\sU}.\]
\end{proof}
\subsection{Proof of Theorem B}
In \cite[Section 1.5]{Tleaves}, J.-H. Lu and V. Mouquin develop a general theory to study the $\mathbb{T}$-leaves of a class of $\mathbb{T}$-Poisson manifolds. Let $G$ be a simply-connected complex semi-simple group.
The general theory is applied to study the $T$-leaves of $(G/B,\pi_{\sG/\sB})$ in \cite[Theorem 1.1]{Tleaves}.
\begin{proposition}\label{TleavesGB}\cite[Theorem 1.1]{Tleaves}
The $T$-leaf decomposition of $(G/B,\pi_{\sG/\sB})$ coincides with the disjoint union of open Richardson varieties
\[G/B=\bigsqcup_{v\leq u}(BuB/B\cap B_-vB/B).\]
\end{proposition}

However, the general theory cannot be applied to study the $T$-leaves of $(G/U,\pi_{\sG/\sU})$. Based on the work in \S\ref{T-mixed}, we give a proof of Theorem B.
First, we recall the following Proposition \ref{Tmix} proved by J.-H. Lu and Y. Mi in \cite[Lemma 2.23]{Integrablesystems}.
\begin{proposition}\label{Tmix}
Let $\mathbb T$ be a complex torus and $(X,\pi_X)$ a complex $\mathbb T$-Poisson manifold. Let $\ft$ be the Lie algebra of $\mathbb T$ and $A\in \ft\otimes \ft$. With respect to the diagonal action of $\mathbb T$ on $\mathbb T\times X$ , the $\mathbb T$-leaves of the
Poisson structure $0\bowtie_AX$ in $\mathbb T\times X$ are precisely all the $\mathbb T\times L$, where $L$ is a $\mathbb T$-leaf of $\pi_X$ in $X$.
\end{proposition}

\begin{theorem}\label{TheoremC}
The decomposition of $G/U$ into $T$-leaves of the Poisson structure $\pi_{\sG/\sU}$ is
\begin{equation}\label{TleaveGU}
G/U=\bigsqcup_{v\leq u}(BuB/U\cap B_-vB/U).
\end{equation}
\end{theorem}
\begin{proof}
Recall that
\[G/U=\bigsqcup_{u\in W} BuB/U\]
such that $BuB$ is a Poisson subvariety of $(G/U,\pi_{\sG/\sU})$ for each $u\in W$. To study the $T$-leaf decomposition of $(G/U,\pi_{\sG/\sU})$,
it is sufficient to study the $T$-leaves of $(BuB/U,\pi_{\sG/\sU})$ for each fixed $u\in W$. It follows from Proposition \ref{mixedGU}, Proposition \ref{TleavesGB} and Proposition \ref{Tmix} that for each $v\leq u\in W$, the subvariety
\[\kappa_{\dot w}^{-1}(T\times (BuB/B\cap B_-vB/B))=BuB/U\cap B_-vB/U\]
is a single $T$-leaf of $(BuB/U,\pi_{\sG/\sU})$. Therefore, (\ref{TleaveGU}) gives $T$-leaf decomposition of $(G/U,\pi_{\sG/\sU})$.
\end{proof}
\sectionnew{General construction of Frobenius splittings via Poisson $\mathbb{T}$-Pfaffians}
\subsection{Poisson algebras in positive characteristics}\label{ppc}
Denote by $k$ an algebraically closed field with ${\rm char}(k) = p > 2$.
The purpose of this section is to explain a few aspects of Poisson algebras over $k$,
and in particular the notion of Poisson $\mathbb{T}$-Pfaffians, $\mathbb{T}$ being an algebraic torus,
which in the case of characteristic $0$
is introduced in \cite{Tleaves}. Note that a special feature
is that, given a Poisson algebra $A$ and any element $a \in A$, the ideal $\langle a^p \rangle$ is a Poisson ideal since
\[
\{a^pb,c\}=a^p\{b,c\}+ pa^{p-1}b\{a,c\}=a^p\{b,c\} \in \langle a^p \rangle, \hs b, c \in A.
\]
\begin{notation}
{\rm	Let $A$ be a $k$-algebra and $n$ be a non-negative integer. Recall that an $n$-derivation on $A$ is an element $\alpha \in \rm{Hom}_k(\wedge^n A, A)$ such that $\alpha$ is a derivation in each of its arguments. Following the notation in \cite[Section 3.1]{CAP}, the vector space consisting of all $n$-derivations on $A$ will be denoted by $\mathfrak{X}^{n}(A)$. By convention, $\mathfrak{X}^{0}(A)=A$.}
\end{notation}

There is a natural $A$-module structure on $\mathfrak{X}^{n}(A)$: for $b \in A$ and $\alpha \in\mathfrak{X}^{n}(A)$,
\[
(b \alpha) (a_1, \ldots, a_n )=b (\alpha (a_1, \ldots, a_n)), \ \ a_1, \ldots, a_n \in A.
\]
Suppose $A$ is finitely generated by $r$ elements. By the skew-symmetry of multi-derivations, $\mathfrak{X}^{n}(A)=0$ for all $n > r$. Hence there exists a largest $n$ such that $\mathfrak{X}^{n}(A) \neq 0$. The module of top degree multi-derivations on a finitely generated $k$-algebra $A$, denoted by $\mathfrak{X}^{top}(A)$, is then defined to be $\mathfrak{X}^{n}(A)$ for such largest $n$.

We now briefly recall the definition of the wedge product and the Schouten bracket on $\mathfrak{X}(A)= \bigoplus_{n=0}^{\infty} \mathfrak{X}^{n}(A)$. Let $\alpha \in \mathfrak{X}^{m}(A)$, $\beta \in \mathfrak{X}^{n}(A)$ and $a_1, \ldots, a_{m+n} \in A$. Then $\alpha \wedge \beta \in \mathfrak{X}^{m+n}(A)$ is defined by
\[
(\alpha \wedge \beta)(a_1, \ldots , a_{m+n}) = \sum_{\sigma \in S_{m,n}} {\rm sgn} (\sigma) \alpha(a_{\sigma(1)}, \ldots, a_{\sigma (m)}) \beta (a_{\sigma(m+1)}, \ldots, a_{\sigma (m+n)}),
\]
where $S_{m,n}$ denotes the set of all $(m,n)$-shuffles, while $[\alpha,\beta]\in \mathfrak{X}^{m+n-1}(A)$ is
defined by
\[
\begin{aligned}
& [\alpha, \beta](a_1, \ldots, a_{m+n-1}) \\
= &\sum_{\sigma \in S_{n,m-1}} {\rm sgn} (\sigma) \alpha (\beta (a_{\sigma(1)}, \ldots, a_{\sigma(n)} ) , a_{\sigma (n+1)} , \ldots , a_{\sigma(n+m-1)} ) \\
& - (-1)^{(m-1)(n-1)} \sum_{\sigma \in S_{m,n-1}} {\rm sgn} (\sigma) \beta (\alpha (a_{\sigma(1)}, \ldots, a_{\sigma(m)} ) , a_{\sigma (m+1)} , \ldots , a_{\sigma(m+n-1)} ).
\end{aligned}
\]
Being a bi-derivation, any Poisson bracket on $A$ is of the form
\[
\{a, \, b\} = \pi(a, b), \hs a, b \in A,
\]
for a unique $\pi \in \fX^2(A)$, and the condition ${\rm char}(k) \neq 2$ implies that the Jacobi identity of
$\{- , - \}$ is equivalent to  $[\pi,\pi]=0$.
\begin{notation}
{\rm For any integer $r \geq 1$, let
	\[
	S_{[2r]}=\{ \sigma \in S_{2r}: \sigma(1) < \sigma(3) < \ldots < \sigma(2r-1) \mbox{ and } \sigma(2i-1) < \sigma(2i), \,
	i = 1, \ldots, r \},
	\]
	and for $\pi \in \fX^2(A)$, set $\pi^{[r]} \in \fX^{2r}(A)$ by
	\[
	\pi^{[r]}(a_1, \ldots, a_{2r})=  \sum_{\sigma \in S_{[2r]}} ({\rm sgn} \sigma) \pi(a_{\sigma(1)}, a_{\sigma(2)}) \ldots \pi(a_{\sigma(2r-1)}, a_{\sigma(2r)}).
	\]}
\end{notation}

\begin{lemma}\label{le-pi-r}
	For any $\pi \in \fX^2(A)$ and any integer $r \geq 1$, one has $\pi^{\wedge r} = r! \pi^{[r]}$.
\end{lemma}

\begin{proof} Denote by $R^r_2$ the set of all elements $\sigma \in S_{2r}$ such that $\sigma(2i-1) < \sigma (2i)$ for all $i=1, \ldots , r$. In particular, $R^2_2$ is just the set of all $(2,2)$-shuffles. By definition,
	for any $a_1, \ldots, a_{2r} \in A$, one has
	\[
	\pi^{\wedge r} (a_1, \ldots, a_{2r}) = \sum_{\sigma \in R^r_2} ({\rm sgn} \sigma) \pi(a_{\sigma(1)}, a_{\sigma(2)}) \ldots \pi(a_{\sigma(2r-1)}, a_{\sigma(2r)}).
	\]
	For each $\sigma \in R^r_2$, we can reorder the factors in the right hand side of the above
	expression so we may assume that $\sigma(1) < \sigma(3) < \ldots < \sigma(2r-1)$, i.e., $\sigma \in S_{[2r]}$. Thus
	\begin{align*}
	\pi^{\wedge r} (a_1, \ldots, a_{2r})&= r! \sum_{\sigma \in S_{[2r]}} ({\rm sgn} \sigma) \pi(a_{\sigma(1)}, a_{\sigma(2)}) \ldots \pi(a_{\sigma(2r-1)}, a_{\sigma(2r)})\\
	& = r!\pi^{[r]}(a_1, \ldots, 2_{2r}).
	\end{align*}
\end{proof}

\begin{definition}
{\rm	The rank of a Poisson structure $\pi$ on $A$, denoted by $\rank (\pi)$, is defined to be $2r$, where $r$ is the largest integer such that $\pi^{[r]} \neq 0$.}
\end{definition}

Let  $A$ be a finitely generated $k$-algebra and
let $\mathfrak{m}$ be a maximal ideal of $A$. Recall that $\mathfrak{m} / \mathfrak{m}^2$ is a vector space over
the field $A / \mathfrak{m}$. As $A$ is finitely generated and  $k$ is algebraically closed,
$A / \mathfrak{m} \cong k$ \cite[Corollary 7.10]{atiyah-macdonald}. Let $\alpha \in \mathfrak{X}^n (A)$ and $a_1, \ldots, a_n \in A$. Suppose $a_i \in \mathfrak{m}^2$ for some $i$. Then by the Leibniz rule, one see that $\alpha (a_1, \ldots, a_n)$ lies in $\mathfrak{m}$. Therefore, $\alpha$ induces an alternating $k$-linear map
\[
\alpha_{\mathfrak{m}} : \wedge^n (\mathfrak{m} / \mathfrak{m}^2) \rightarrow A/ \mathfrak{m}.
\]
Assume now that $A$ is a Poisson $k$-algebra with Poisson structure $\pi \in \fX^2(A)$. Then $\pi$
induces an alternating $k$-linear map for each maximal ideal $\mathfrak{m}$ of $A$
\[
\pi_{\mathfrak{m}} :  (\mathfrak{m} / \mathfrak{m}^2) \wedge (\mathfrak{m} / \mathfrak{m}^2) \rightarrow k.
\]
Consequently, one has the skew-symmetric linear transformation
\[
\pi^{\sharp}_{\mathfrak{m}} :  \mathfrak{m} / \mathfrak{m}^2 \rightarrow ( \mathfrak{m} / \mathfrak{m}^2 )^{*}.
\]
of $k$-vector spaces for each maximal ideal $\mathfrak{m}$ of $A$.

\begin{definition}
{\rm	The rank of a Poisson structure $\pi$ at a maximal ideal $\mathfrak{m}$ of $A$, denoted by $\rank (\pi_{\mathfrak{m}})$, is defined to be the rank of $\pi^{\sharp}_{\mathfrak{m}} $. It is an even integer as we are assuming that ${\rm char}(k) \neq 2$.}
\end{definition}

Let $V$ be a $k$-vector space and let $\omega \in \wedge^2 V$. Similar to the case of multi-derivation, we have the element $\omega^{[l]} \in \wedge^{2l} V$. Following previous notation, define $\omega^{[l]} \in \wedge^{2l} V$ by
\[
\omega^{[l]}(v_1, \ldots, v_{2l})=  \sum_{\sigma \in S_{[2l]}} ({\rm sgn} \sigma) \omega(v_{\sigma(1)}, v_{\sigma(2)}) \ldots \omega(v_{\sigma(2l-1)}, v_{\sigma(2l)}), \ \ v_i \in V^{*}.
\]

\begin{remark}
	{\rm It is easy to see that $(\pi^{[l]} )_{\mathfrak{m}} = (\pi_{\mathfrak{m}})^{[l]} \in \wedge^{2l} ( \mathfrak{m} / \mathfrak{m}^2 )^{*}$. Therefore there is no ambiguity to write $\pi^{[l]}_m$.}
\end{remark}

\begin{proposition}
	Let $\pi$ be a Poisson structure on $A$. Then $\rank(\pi) \geq \rank (\pi_{\mathfrak{m}})$ for every maximal ideal $\mathfrak{m}$. Moreover, if $A$ is reduced, then there exists some maximal ideal $\mathfrak{m}$ such that $\rank(\pi) = \rank (\pi_{\mathfrak{m}}) $.
\end{proposition}

\begin{proof}
	Suppose $\rank (\pi)=2r$. Then $\pi^{[ r+1]} =0$, hence $\pi^{[r+1]}_{\mathfrak{m}} =0 $, which implies that $\rank (\pi_{\mathfrak{m}}) \leq 2r$.
	
	To prove the second part, we assume the contrary. Suppose $\rank (\pi_\mathfrak{m}) < 2r$ for every maximal ideal $\mathfrak{m}$, so that $\pi^{[r]}_\mathfrak{m} =0$. By definition, it means that $\pi^{[r]}_\mathfrak{m} (b_1, \ldots , b_{2r})=0$ for all $b_i \in \mathfrak{m}$. In other words, one has $\pi^{[r]}(b_1, \ldots, b_{2r}) \in \mathfrak{m}$ for all $b_i \in \mathfrak{m}$.
	Fix any maximal ideal $\mathfrak{m}$, one has $A=k \bigoplus \mathfrak{m}$. Let $a_1, \ldots, a_{2r} \in A$ and write $a_i=c_i+b_i$, where $c_i \in k$ and $b_i \in \mathfrak{m}$. Then
	\[
	\pi^{[r]}(a_1, \ldots, a_{2r}) = \pi^{[r]} (c_1+b_1 , \ldots, c_{2r}+b_{2r}) =\pi^{[r]}(b_1, \ldots, b_{2r}) \in \mathfrak{m}.
	\]
	This holds for every maximal ideal $\mathfrak{m}$. Therefore $\pi^{[r]}(a_1, \ldots, a_{2r}) $ lies in the Jacobson radical of $A$. Since $A$ is a finitely generated $k$-algebra, its Jacobson radical equals its nil-radical, which is zero by our assumption that $A$ is reduced. So $\pi^{[r]}=0$, contradicting that the rank of $\pi$ is $2r$.
\end{proof}

\begin{example}
{\rm	Let $A=k[x,y,z]/I$, where $I=\langle x^2\rangle$. Then $A$ is not reduced since $x$ is a nilpotent element. Let $\pi \in \mathfrak{X}^2(A)$ be defined by
	\[
	\pi(x,y)=x, \ \pi(x,z)=x, \ \pi(y,z)=0.
	\]
	Then for any $f \in k[x,y,z]$, one has $\pi(x,f) \in \langle x \rangle$. Therefore $\pi$ is well-defined because $\pi(x^2,f)=2x \pi(x,f) \in I$. Since $\pi \neq 0$, one must have $\rank(\pi) > 0$. Actually it can be easily seen that $\rank(\pi)=2$.
	On the other hand, maximal ideals of $A$ are of the form $\mathfrak{m}=\langle x,y-y_0,z-z_0\rangle$, where $y_0,z_0 \in k$. Fix arbitrary $y_0$ and $z_0$. Then $\pi_m$ is given by
	\[
	\begin{aligned}
	&\pi_m(x,y-y_0)=\pi(x,y-y_0)=\pi(x,y)=x \in m,\\
	&\pi_m(x,z-z_0)=\pi(x,z-z_0)=\pi(x,z)=x \in m,\\
	&\pi_m(y-y_0,z-z_0)=\pi(y,z)=0.
	\end{aligned}
	\]
	Hence $\pi_\mathfrak{m}$ is trivial, implying that $\rank(\pi_\mathfrak{m})=0$, for any maximal ideals of $A$.}
\end{example}

\subsection{Frobenius Poisson $\mathbb{T}$-Pfaffians}

Recall that an algebraic torus $\mathbb{T}$ over $k$ is an algebraic group isomorphic to $(k^{*})^n$ for some non-negative integer $n$.
A character of $\mathbb{T}$ is an algebraic group homomorphism $\lambda : \TT \rightarrow k^*$. The set of all characters of $\TT$ forms an abelian group with the group operation being pointwise multiplication and is denoted by $X^{*}(\TT)$, called the character group of $\TT$.
Any identification $\TT \cong (k^{*})^n$ gives coordinates $(t_1, \ldots, t_n)$ on $\TT$. Let
$\lambda \in X^*(\TT)$. Then
$\lambda$ is a morphism of varieties, so $\lambda(t_1, \ldots, t_n)$ should be a Laurent monomial in these variables. As $\lambda$ is also a group homomorphism, $\lambda$ should be a Laurent monomial. Hence
\displaymath
\lambda(t_1, \ldots, t_n)={t_1}^{r_1} \ldots {t_n}^{r_n}
\enddisplaymath
for some $r_1, \ldots, r_n \in \mathbb{Z}$. On the other hand, it can be easily seen that all $\lambda$ of the form above are characters of $\TT$.  One thus sees that $X^{*}(\TT)$ is isomorphic to $\mathbb{Z}^n$, a free abelian group of rank $n$.
For the rest of the chapter, we will identify $X^{*}(\TT)$ with $\mathbb{Z}^n$. Also, we write $t^{\lambda}$ for $\lambda(t)$.

Suppose $\lambda \in X^{*}(\TT)$, then it induces a map $\lambda: \mathfrak{t} \rightarrow k$, where $\mathfrak{t}$ is the Lie algebra of $\mathbb{T}$, given by
\[
\lambda(h)=\sum_{i=1}^n \lambda_i (h_i),
\]
where $\lambda=(\lambda_1, \ldots , \lambda_n )$ and $h=(h_1 ,\ldots, h_n)$.

Suppose now $\TT$ acts on a vector space $V$. Let $v \in V$. If there exists some function $\lambda: \TT \rightarrow k^{*}$ such that $tv=\lambda (t)v$ for all $t \in \TT$, we say that $v$ is a weight vector with weight $\lambda$. The set of all $\lambda$-weight vectors form a vector subspace of $V$, called the weight space and is denoted by $V_{\lambda}$.

\begin{definition}
	{\rm Let $A$ be a $k$-algebra with a $\TT$-action. The action is said to be rational if $A$ can be decomposed into a direct sum of weight spaces, i.e.
	\[
	A=\bigoplus_{\lambda \in X^{*}(T)} V_{\lambda}.
	\]
	Note that the coordinate ring of an affine $\TT$-variety over $k$ is an algebra with a rational $\TT$-action.
    }
\end{definition}

\begin{definition}\label{gvf}
{\rm Suppose $\TT$ acts on $A$ rationally. Let $a \in A$. From the weight-space decomposition of $A$, we can write $a=\sum_{i=1}^n a_i$, where $a_i \in V_{\lambda_i}$ for some $\lambda_i \in X^{*}(\TT)$. Let $h \in \mathfrak{t}$, the $k$-derivation on $A$ generated by $h$, denoted by $\partial_h$, is defined as
\[
\partial_h (a)= \sum_{i=1}^n \lambda_i (h)a_i.
\]
}
\end{definition}

\begin{definition} {\rm Let $A$ be a Poisson $k$-algebra.
	A Poisson $\TT$-action on $A$ is a $\TT$-action on $A$ by Poisson automorphisms. A rational Poisson $\TT$-action is a rational $\TT$-action on $A$ which is at the same time a Poisson $\TT$-action.}
\end{definition}

\begin{lemma}\label{9}
	For $h \in \mathfrak{t}$, the derivation $\partial_h$ induced from a rational Poisson torus action is a Poisson derivation.
\end{lemma}

\begin{proof}
	Let $a=\sum_{i=1}^n a_i$ and $b=\sum_{j=1}^m b_i$. One has $\{a_i, b_j \} \in V_{\lambda_i + \mu_j}$, therefore
	\[
	\partial_h \{ a , b \}= \partial_h ( \sum_{i,j} \{ a_i ,b_j \} )= \sum_{i,j} ( \lambda_i +\mu_j ) (h) \{ a_i ,b_j \}.
	\]
	On the other hand,
	\begin{eqnarray*}
		&& \{ \partial_h (a), b \}+ \{a, \partial_h (b) \}\\
		&=&\{ \sum_i \lambda_i (h) a_i, \sum_j b_j \}+ \{\sum_i a_i , \sum_j \mu_j (h) b_j \} \\
		&=&\sum_{i,j} \lambda_i (h)  \{ a_i, b_j \}+ \sum_{i,j} \mu_j (h)  \{a_i , b_j \} \\
		&=&\sum_{i,j} ( \lambda_i +\mu_j ) (h) \{ a_i ,b_j \}.
	\end{eqnarray*}
\end{proof}

\begin{definition}
	{\rm Let $(A,\pi)$ be a finitely generated Poisson algebra with a rational Poisson $\TT$-action. Suppose the rank of the Poisson structure on $A$ is $2r$.\emph{ A Poisson $\TT$-Pfaffian }of $(A,\pi)$ is a non-zero element $\sigma \in \mathfrak{X}^{top}(A)$ of the form
	\[
	\sigma={\pi}^{[r]} \wedge \partial_{h_1} \wedge \ldots \wedge \partial_{h_l},
	\]
	where $h_1 , \ldots , h_l \in \mathfrak{t}$.}
\end{definition}

For $a \in A$, define $X_a \in \mathfrak{X}^{1}(A)$ by $X_a(b) = \{a, b\}$ for $b \in A$. It is easy to see that
$X_a$ is a Poisson derivation of $A$. We will refer to $X_a$ as the {\it Hamiltonian derivation} of $a$.

\begin{proposition}\label{10}
	The following formula holds for all positive integers $l$ and all $a_1, \ldots ,a_{2l} \in A$:
	\displaymath
	\pi^{[ l]} (a_1 , \ldots, a_{2l}) = -(\pi^{[l-1]} \wedge X_{a_{2l} }) (a_1, \ldots , a_{2l-1}).
	\enddisplaymath
	In particular, $\pi^{[r]} \wedge X_a=0$ for the Hamiltonian derivation $X_a$, where $r$ is the half rank of $\pi$.
\end{proposition}

\begin{proof}
	By definition, one has
	\[
	\begin{aligned}
	& (\pi^{[l-1]} \wedge X_{a_{2l} }) (a_1, \ldots , a_{2l-1}) \\
	=& \sum_{\sigma \in S_{2l-2,1}} ({\rm sgn} \sigma) \pi^{[l-1]} (a_{\sigma(1)} , \ldots, a_{\sigma (2l-2)  }) X_{a_{2l} } (a_{\sigma (2l-1)}) \\
	=& - \sum_{\sigma \in S_{2l-2,1}} ({\rm sgn} \sigma) \pi^{[l-1]} (a_{\sigma(1)} , \ldots, a_{\sigma (2l-2)  }) \pi (a_{\sigma (2l-1)} , a_{2l}) \\
	=& - \sum_{ \substack{  \sigma \in S_{2l-2,2} \\ \sigma(2l)=2l } }
	({\rm sgn} \sigma) \pi^{[l-1]} (a_{\sigma(1)} , \ldots, a_{\sigma (2l-2)  }) \pi (a_{\sigma (2l-1)} , a_{\sigma(2l)}).
	\end{aligned}
	\]
	For each $\sigma \in S_{2l-2,2}$ with $\sigma(2l)=2l$, we know that
	\displaymath
	\begin{aligned}
		& ({\rm sgn} \sigma) \pi^{[l-1]} (a_{\sigma(1)} , \ldots, a_{\sigma (2l-2)  }) \pi (a_{\sigma (2l-1)} , a_{\sigma(2l)}) \\
		=& \sum_{\tau} {\rm sgn} ( \tau \sigma )  \pi(a_{\tau \sigma (1)},a_{\tau \sigma (2)}) \ldots \pi (a_{\tau \sigma (2l-3)},a_{\tau \sigma (2l-2)}) \pi (a_{\sigma (2l-1)} , a_{\sigma(2l)}),
	\end{aligned}
	\enddisplaymath
	where $\tau$ runs through all permutations on $\{\sigma(1) , \ldots, \sigma(2l-2) \}$ such that $\tau \sigma(1) < \tau \sigma (3) < \ldots < \tau \sigma (2l-3)$ and $\tau \sigma (2i-1) <  \tau \sigma (2i)$ for $i=1, \ldots, l-1$.
	If we extend each $\tau $ to a permutation in $S_{2l}$ by setting $\tau\sigma(2l-1)=\sigma(2l-1)$ and $\tau \sigma(2l)=\sigma (2l)=2l$, then
	\displaymath
	(\pi^{[l-1]} \wedge X_{a_{2l} }) (a_1, \ldots , a_{2l-1}) = -\sum_{\tau \sigma} {\rm sgn} ( \tau \sigma )  \pi(a_{\tau \sigma (1)},a_{\tau \sigma (2)}) \ldots \pi (a_{\tau \sigma (2l-1)} , a_{\tau \sigma(2l)}),
	\enddisplaymath
	where $\tau \sigma $ runs through all permutations in $R^l_2$ such that $\tau \sigma(1) < \tau \sigma (3) < \ldots < \tau \sigma (2l-3)$ and $\tau \sigma(2l)=2l$.
	Recall that $S_{[2r]}$ is the subset of permutations in $R^l_2$ such that $\sigma(1)< \sigma(3) \ldots <\sigma(2l-1)$. By rearranging the order of the terms, we obtain
	\[
	(\pi^{[l-1]} \wedge X_{a_{2l} }) (a_1, \ldots , a_{2l-1})=\sum_{  \sigma \in S_{[2r]} }
	({\rm sgn} \sigma) \pi (a_{\sigma(1)}, a_{\sigma(2)}) \ldots \pi (a_{\sigma(2l-1)} , a_{\sigma(2l)} ).
	\]
	On the other hand, still by definition
	\[
	\pi^{[l]} (a_1 , \ldots, a_{2l}) =  \sum_{  \sigma \in S_{[2r]} }
	({\rm sgn} \sigma) \pi (a_{\sigma(1)}, a_{\sigma(2)}) \ldots \pi (a_{\sigma(2l-1)} , a_{\sigma(2l)} ).
	\]
	Therefore,
	\[
	\pi^{[l]} (a_1 , \ldots, a_{2l}) = -\pi^{[l-1]} \wedge X_{a_{2l} } (a_1, \ldots , a_{2l-1}).
	\]
\end{proof}

\begin{proposition}\label{hamiltonvanish}
	Let $\sigma$ be a Poisson $\TT$-Pfaffian and $V$ a Hamiltonian derivation or a derivation generated by some element of the Lie algebra of $\mathbb{T}$ for the torus action. One has $[V, \sigma] = 0$.
\end{proposition}
\begin{proof}
	By definition,
	\[
	[V, \sigma ] =  [V ,\pi ] \wedge \pi^{[r-1]} \wedge \partial_{h_1} \wedge \ldots \wedge \partial_{h_l} \\
	+ \sum_{i=1}^l \pi^{[r]} \wedge \partial_{h_1} \wedge \ldots \wedge [V, \partial_{h_i} ] \wedge \ldots \wedge \partial_{h_l}.
	\]
	If $V$ is Hamiltonian, then $[V,\pi]$ is zero. On the other hand, $V$ is Hamiltonian implies $[V,\partial_{h_i}]$ is Hamiltonian. Then $\pi^{[r]} \wedge [V,\partial_{h_i}]$ is also zero by Proposition \ref{10}.	If $V$ is generated by some element of the Lie algebra of $\mathbb{T}$ for the torus action, then both $[V,\pi]$ and $[V,\partial_{h_i}]$ are zero.
\end{proof}

\begin{remark}
{\rm	It is straightforward to see that the results in this section generalize to $\mathbb{T}$-Poisson varieties.}
\end{remark}

\subsection{Rank of Poisson CGL extensions}\label{CGL}
We first recall from \cite[Section 1.3]{PoissonCGL} the definition of Poisson CGL extensions.
\begin{definition}
	\label{1}
	{\rm A Poisson CGL extension is a polynomial algebra $k[x_1, \ldots, x_n]$ with a $\TT$-Poisson structure for which $x_1, \ldots, x_n$ are $\TT$-eigenvectors, where $\TT$ is an algebraic torus, together with elements $h_1, \ldots, h_n \in \mathfrak{t}$ such that the following conditions are satisfied for $1 \leq i \leq n$:
	
	(a) For $i<j$, $\{ x_i , x_j\} = \lambda_j(h_i)x_ix_j+\delta_i(x_j)$ for some $\lambda_i \in \mathfrak{t}^{*}$ and some locally nilpotent derivation $\delta_i$ on $k[x_{i+1}, \ldots, x_n]$ with $\delta_i(x_j) \in k[x_{i+1}, \ldots, x_{j-1}]$.
	
	(b) The $h_i$-eigenvalue of $x_i$ is non-zero, i.e. $\lambda_i(h_i) \neq 0$.}
\end{definition}

\begin{remark}
{\rm	A Poisson CGL extension is essentially an affine space equipped with a special $\mathbb{T}$-Poisson structure. Thus the notions of Poisson $\mathbb{T}$-Pfaffians and Frobenius Poisson $\mathbb{T}$-Pfaffians make sense for Poisson CGL extensions. In this case, the vector field $\partial_{h}$ generated by an element $h \in \mathfrak{t}$ is given by
	\[
	\partial_h = \lambda_1(h)x_1\partial_1 + \lambda_2(h)x_2\partial_2 + \ldots + \lambda_n(h)x_n\partial_n.
	\]
}
\end{remark}

\begin{definition}\label{PCGLV}
{\rm
The Poisson bivector of a Poisson CGL extension, which is written as
\[
\pi=\sum_{i<j}( \lambda_j (h_i)x_ix_j + \delta_i (x_j) ) \partial_i \wedge \partial_j,
\]
is called \emph{a Poisson CGL bivector field}.}
\end{definition}
\begin{remark}\label{lc}
{\rm
The bivector field defined by
\[
\pi_0= \sum_{i<j} \lambda_j (h_i) x_i x_j \partial_i \wedge \partial_j.
\]
is a Poisson bivector field, which is called \emph{the log-canonical part} of $\pi$.
The bivector field $\sum_{i<j}(  \delta_i (x_j) ) \partial_i \wedge \partial_j$ is not necessarily a Poisson vector field, called\emph{ the $\delta$-part} of $\pi$.}
\end{remark}

\begin{proposition}
	\label{5}
	The rank of the Poisson bivector $\pi$ of a Poisson CGL extension is greater than or equal to that of $\pi_0$.
\end{proposition}
\begin{proof}
	Suppose $\pi^{[ r]}=0$ for some non-negative integer $r$. Let $w=\pi - \pi_0 = \sum_{i<j} \delta_i (x_j) \partial_i \wedge \partial_j$. Then
	\[
	\pi^{[r]}= \pi_0^{[r]} + \sum_{l=0}^{r-1}  \pi_0^{[l]} \wedge w^{[r-l]}.
	\]
	Let $1 \leq  i_1 <  \ldots < i_{2r } \leq n$. For $0 \leq l \leq r-1$, one has
	\[
	\begin{aligned}
	&\pi_0^{[l]} \wedge w^{[r-l]} (dx_{i_1} \wedge \ldots \wedge dx_{i_{2r}}) \\
	=&\sum_{\sigma \in S_{2l,2r-2l}} \pi_0^{[j]}(dx_{i_{\sigma(1)}} \wedge \ldots \wedge dx_{i_{\sigma(2l)}}) w^{[r-j]} (dx_{i_{\sigma(2l+1)}} \wedge \ldots \wedge dx_{i_{\sigma(2r)}}).
	\end{aligned}
	\]
	For each $\sigma \in S_{2l,2r-2l}$, we know that $\pi_0^{[l]}(dx_{i_{\sigma(1)}} \wedge \ldots \wedge dx_{i_{\sigma(2l)}})$ is a scalar multiple of $x_{i_{\sigma(1)}} \ldots x_{i_{\sigma(2l)}}$. However $w^{[r-l]} (dx_{i_{\sigma(2l+1)}} \wedge \ldots \wedge dx_{i_{\sigma(2r)}})$ does not involve the variable $x_{i_{\sigma(2r)}}$ since $\delta_i(x_j)$ does not involve the variable $x_j$ for every $i < j$. Therefore the coefficient of $x_{i_{1}} \ldots x_{i_{2r}} $ in $\pi_0^{[l]} \wedge w^{[r-l]} (dx_{i_1} \wedge \ldots \wedge dx_{i_{2r}})$ is zero. As $\pi_0(dx_{i_1} \wedge \ldots \wedge dx_{i_{2r}})$ is a scalar multiple of $x_{i_{1}} \ldots x_{i_{2r}}$, our assumption implies that $\pi_0(dx_{i_1} \wedge \ldots \wedge dx_{i_{2r}})=0$ for every $1 \leq  i_1 <  \ldots < i_{2r } \leq n$. Hence $\pi_0^{[r]}=0$.
\end{proof}

Let $(k[x_1, \ldots, x_n], \pi)$ be a Poisson CGL extension.  Let $e_i=x_i \partial_i$ and $V$ be the $n$-dimensional $k$-vector space with basis $\{ e_1, \ldots, e_n \}$. Then $\pi_0$, the log-canonical term of $\pi$, is represented by the skew-symmetric matrix $\Omega-\Omega^T$, where $\Omega$ is the $n \times n$ strictly upper triangular matrix with entries $\lambda_i (h_j)$ for $i<j$. Suppose the rank of the matrix $\pi_0$ is $2r$, let $v_1, \ldots, v_{n-2r}$ be $n-2r$ elements in $V$. Note that with the basis $e_i$'s, the $v_i$'s can be regarded as size $n$ column vectors. One can define an $n \times (2n-2r)$-matrix $M$ as
\[
M= \begin{pmatrix}
\Omega-\Omega^T &v_1 & v_2 & \ldots & v_{n-2r}  \\
\end{pmatrix}.
\]

\begin{lemma}\label{fullrank}
	The element $\pi_0^{[r]} \wedge v_1 \wedge \ldots \wedge v_{n-2r} \in \wedge^{n} V$ is non-zero if and only if $M$ is of full rank.
\end{lemma}

\begin{proof}
	Denote the columns of $\pi_0$ by $w_1, \ldots, w_n$. Since the rank of $\pi_0$ is $2r$, without loss of generality, we may assume that the image of $\pi_0$ is $Span\{w_1, \ldots, w_{2r} \}$. Since $\pi_0 \in \wedge^2  \ im(\pi_0)$, $\pi_0^{[r]}$ is a nonzero multiple of $w_1 \wedge \ldots \wedge w_{2r}$. Hence $\pi_0^{[r]} \wedge v_1 \wedge \ldots \wedge v_{n-2r} \neq 0$ if and only if $w_1, \ldots, w_{2r} , v_1, \ldots, v_{n-2r}$ are linearly independent, which is equivalent to saying that $M$ is of full rank.
\end{proof}

\begin{lemma}\label{polydegree}
	Suppose $\pi^{[r]} \wedge \partial_{h_{i_1}} \wedge \ldots \wedge \partial_{h_{i_{n-2r}}} = f \partial_{1} \wedge \ldots \wedge \partial_{n}$ is non-zero for some integers $\{i_1, \ldots, i_{n-2r} \} \subset [1,n]$, where $f \in k[x_1, \ldots, x_n]$. Let
	$g = c x_1^{t_1}\ldots x_n^{t_n}$ be a non-zero monomial term in $f$, where $c \in k$. Then either $(t_1, t_2, \ldots, t_n) = (1, 1, \ldots, 1)$ or there exists some integer $s$ in $[1, n]$ such that $t_s = 0$ and $t_{j} \leq 1$ for all $j > s$.
\end{lemma}

\begin{proof}
	Following the notation in the proof of Proposition \ref{5}, write
	\[
	\pi^{[r]} \wedge \partial_{h_{i_1}} \wedge \ldots \wedge \partial_{h_{i_{n-2r}}}=  \sum_{l=0}^{r}  \pi_0^{[l]} \wedge w^{[r-l]} \wedge \partial_{h_{i_1}} \wedge \ldots \wedge \partial_{h_{i_{n-2r}}}.
	\]
	Then the pairing with $dx_1 \wedge \ldots \wedge dx_n$ of each term in the expansion of the expression
	\[
	\pi_0^{[l]} \wedge w^{[r-l]} \wedge \partial_{h_{i_1}} \wedge \ldots \wedge \partial_{h_{i_{n-2r}}}
	\]
	is equal to a scalar multiple of
	\begin{equation}\label{threeproduct}
	\left (\prod_{j = 1}^{l} \lambda_{\sigma(2j - 1)}(h_{\sigma(2j)})x_{\sigma(2j-1)}x_{\sigma(2j)}\right ) \left ( \prod_{j = l+1}^{r} \delta_{\sigma(2j - 1)}(x_{\sigma(2j)}) \right ) \left (\prod_{j = 1}^{n - 2r} \lambda_{\sigma(2r + j)}(h_{i_j})x_{\sigma(2r + j)} \right )
	\end{equation}
	for some $\sigma \in S_n$ such that $\sigma(2j - 1) < \sigma(2j)$ for $j = l + 1, l+2, \ldots, r$.
	To prove the statement, it suffices to prove that when $g = c x_1^{t_1}\ldots x_n^{t_n}$ is a non-zero monomial term in the product above, where $c \in k$, then either $(t_1, t_2, \ldots, t_n) = (1, 1, \ldots, 1)$ or there exists some integer $s$ in $[1, n]$ such that $t_s = 0$ and $t_{j} \leq 1$ for all $j > s$.
	
	In the case when $l = r$, it is obvious $t_1 = \ldots = t_n = 1$.
	In the case when $l < r$, choose $j$ in $l+1,l+2,\ldots, r$ so that $\sigma(2j)$ is maximal among $\sigma(2l + 2), \sigma(2l+4), \ldots, \sigma(2r)$. Let $s$ be $\sigma(2j)$. Then in the second bracket in expression \ref{threeproduct}, the degrees of $x_s, x_{s+1}, \ldots, x_n$ are all zero.
	In the product of the first and the third bracket in expression \ref{threeproduct}, the degree of $x_s$ is zero and the degrees of $x_{s+1}, \ldots, x_n$ are at most one.
\end{proof}

\subsection{Proof of Theorem C}
For the convenience of the readers, we first restate the Theorem C.
\begin{theorem}\label{A}
If the rank of a Poisson CGL bivector field $\pi_X$ is equal to that of $(\pi_X)_0$, the log-canonical part of $\pi_X$, then a Frobenius Poisson $\TT$-Pfaffian exists.
\end{theorem}
\begin{proof}
Suppose $\pi^{[r]} \wedge \partial_{h_{i_1}} \wedge \ldots \wedge \partial_{h_{i_{n-2r}}} = f \partial_{1} \wedge \ldots \wedge \partial_{n}$ is non-zero. Then by \cite[Theorem 1.3.8]{BK}, a non-zero scalar multiple of it defines a Frobenius splitting if and only if the coefficient of $(x_1\ldots x_n)^{p-1}$ in $f^{p-1}$ is a non-zero constant. By Lemma \ref{polydegree}, we know that the coefficient of $(x_1\ldots x_n)^{p-1}$ in $f^{p-1}$ is equal to the $p-1$ power of the coefficient of $x_1\ldots x_n$ in $f$. Write $\pi_0^{[r]} \wedge \partial_{h_{i_1}} \wedge \ldots \wedge \partial_{h_{i_{n-2r}}} = g \partial_{1} \wedge \ldots \wedge \partial_{n}$. From Proposition \ref{5} we see that the coefficient of $x_{j_1} \ldots x_{j_{2r}}\partial_{j_1} \wedge \ldots \wedge \partial_{j_{2r}}$ in $\pi^{[r]}$ is equal to that in $\pi_0^{[r]}$ for all $\{ j_1, \ldots , j_{2r} \}$. Therefore the coefficient of $x_1 \ldots x_n$ in $f$ is equal to that in $g$. But $g$ is just a scalar multiple of $x_1\ldots x_n$. So to prove the statement, we only need to prove that $\pi_0^{[ r]} \wedge \partial_{h_{i_1}} \wedge \ldots \wedge \partial_{h_{i_{n-2r}}}$ is zero.
	
	We know that $\partial_{h_j}$ is represented by the size $n$ column vector with the $i$-th entry being $\lambda_i (h_j)$. Denote the $n \times n$ strictly upper triangular matrix with $i,j$-th entry being $\lambda_i (h_j)$, where $i<j$, by $\Omega$. Denote the $n \times n$ strictly lower triangular matrix with $i,j$-th entry being $\lambda_i (h_j)$, where $i>j$, by $L$. Then the matrix representing $(\partial_{h_1} \ \ldots \ \partial_{h_n})$ is written as $\Omega+L+D$, where $D$ is the non-singular diagonal matrix with entries $\lambda_i (h_i)$.
	Therefore, by Lemma \ref{fullrank}, to prove the statement, we only need to prove that the following $n \times 2n$ matrix
	\[
	M= \begin{pmatrix}
	\Omega-\Omega^T & \Omega+L+D  \\
	\end{pmatrix}
	\]
	is of full rank.
	When $n=1$, the conclusion is trivial. Suppose the conclusion holds for a general $n$. For $n+1$, denote the size $n$ column vector with $i$-th entry $\lambda_i (h_{n+1})$ by $v_1$, denote the size $n$ row vector with $i$-th entry $\lambda_{n+1} (h_i)$ by $v_2^T$. Then one has
	\[
	\Omega'= \begin{pmatrix}
	\Omega & v_1  \\
	0 & 0   \\
	\end{pmatrix}
	\  , \
	D'= \begin{pmatrix}
	D & 0  \\
	0 & \lambda_{n+1} (h_{n+1})   \\
	\end{pmatrix}
	\mbox{ and }
	L'= \begin{pmatrix}
	L & 0  \\
	v_2^T & 0   \\
	\end{pmatrix}.
	\]
	Therefore the $(n+1) \times (2n+2)$ matrix $M'$ is
	\[
	M'= \begin{pmatrix}
	\Omega-\Omega^T & v_1 & \Omega+L+D & v_1  \\
	-{v_1}^T & 0 & {v_2}^T & \lambda_{n+1} (h_{n+1})  \\
	\end{pmatrix}.
	\]
	Denote the length $2n$ row vector $( -{v_1}^T \ {v_2}^T )$ by $w$. If we move the $(n+1)$-th column of $M'$ to the last, $M'$ becomes
	\[
	\begin{pmatrix}
	M & v_1 & v_1 \\
	w & \lambda_{n+1} (h_{n+1}) & 0 \\
	\end{pmatrix}.
	\]
	Subtract the last column from the second last column gives,
	\[
	\begin{pmatrix}
	M & 0 & v_1 \\
	w & \lambda_{n+1} (h_{n+1}) & 0 \\
	\end{pmatrix}.
	\]
	By definition we know that $\lambda_{n+1} (h_{n+1}) \neq 0$. Hence the last row of the matrix above cannot be a linear combination of the first $n$ rows since the second last entry is zero for each of the first $n$ rows. Therefore $M'$ is also of full rank.
\end{proof}
\subsection{Proof of Theorem D}
To prepare for the proof of Theorem D, we first cite the following Proposition \ref{finitec} and prove the following Lemma \ref{cs}.

\begin{proposition}\label{finitec}\cite{KM}
	Let $X$ be a non-singular variety split by $\sigma \in H^0(X, \omega_X^{1-p})$. Then for all compatibly split subvarieties $Y$, one has $Y \subset Z(\sigma)$.
\end{proposition}

\begin{lemma}
	\label{cs}
	Let $Y$ be a closed subvariety of a non-singular variety $X$ with ideal sheaf $\mathcal{I}_Y$. Let $V \in H^0(X,\mathcal{T}_X)$. Then $Y$ is invariant under $V$ if and only if $\mathcal{I}_{Y,p}$ is invariant under $V_p$ in $\mathcal{O}_{X,p}$ for every closed point $p$ of some open subset $\mathcal{U}$ of $X$.
\end{lemma}

\begin{proof}
	It suffices to prove the statement for the affine case. Let $X={
\rm Spec}(A)$, where $A$ is a finitely generated $k$-algebra which is also an integral domain. Suppose $Y$ is given by an ideal $I$ of $A$. Suppose in addition that for every $a \in I$, $[V_p, a] \in I_p$ for all $p \in \mathcal{U}$. That means $[V_p,a]$ is zero in the quotient $A_p/I_p \cong (A/I)_p$. Since $\mathcal{U}$ is dense, $[V,a]$ actually vanishes in $A/I$, which means $[V,a] \in I$. The other direction is obvious.
\end{proof}

\begin{theorem}\label{E}
	Let $X$ be a non-singular variety over $k$, $V \in H^0 (X, \mathcal{T}_X)$ and $\sigma \in H^0 (X, \omega_X^{-1})$. Suppose $X$ is split by $\sigma^{p-1}$ and $[V,\sigma]=0$. Then all compatibly split subvarieties of $\sigma^{p-1}$ are invariant under $V$. Moreover, if the Frobenius splitting is defined by a Poisson $\TT$-Pfaffian, then all compatibly split subvarieties are $\TT$-Poisson subvarieties.
\end{theorem}
\begin{proof}
	Let $D$ be a compatibly split divisor. We first prove that $D$ must be invariant under $V$. Since  $D^{reg}$, the smooth locus of $D$, is open in $D$, we only need to prove the statement locally on $D^{reg}$. Let $p$ be a closed point of $D^{reg}$. Choose local parameters $z_1, \ldots, z_n$ of $X$ such that $D$ is locally defined by $z_1=0$. Denote $\langle z_1 \rangle$ by $I$. Then
	\[
	V=\sum_{i=1}^n g_i \frac{\partial}{\partial z_i}, \mbox{ for some $g_i \in O_{X,p} \subset k[[z_1, \ldots, z_n]]$}
	\]
	and
	\[
	\sigma=f \frac{\partial}{\partial z_1} \wedge \ldots  \wedge \frac{\partial}{\partial z_n},\mbox{ for some $f \in O_{X,p} \subset k[[z_1, \ldots, z_n]]$}.
	\]
	By Proposition \ref{finitec}, this implies that $f=t_1^m h$ for some $h$ and $m \geq 1$. Here we choose $m$ to be maximal. Since $\sigma$ defines a Frobenius splitting, $m$ cannot be a multiple of $p$. On the other hand, by direct calculation, $[V,\sigma]=0$ implies
	\[
	0 =\sum_{i=1}^n \frac{\partial}{\partial z_i} (g_if) = \sum_{i=1}^n \frac{\partial}{\partial z_i} (g_iz_1^mh)
	=mg_1hz_1^{m-1} + z_1^m\sum_{i=1}^n \frac{\partial}{\partial z_i} (g_ih).
	\]
	Therefore $g_1h \in I$. Because $m$ is chosen to be maximal, it follows that $g_1 \in I$. Let $l \in O_{X,p}$, then
	\[
	[V,z_1l]=\sum_{i=1}^n g_i \frac{\partial}{\partial z_i}(z_1l)=g_1l + \sum_{i=1}^n z_1g_i \frac{\partial l}{\partial z_i} \in I.
	\]
	Hence $D$ is invariant under $V$.
	
	Next we prove the statement by induction on the dimension of $X$. The case where ${\rm dim}(X)=1$ follows from above. Suppose the statement holds for all $X$ with ${\rm dim}(X)<n$, we consider the case where ${\rm dim}(X)=n$. By our assumption, $X$ is split by a $p-1$ power. The compatibly split codimension $1$ subvarieties are exactly the irreducible components of $Z(\sigma)$, say, $D_1, \ldots, D_m$. Let $Y$ be a compatibly split subvariety, then $Y \subset \bigcup D_i$. Let $Y_i$ be the scheme-theoretic intersection of $Y$ and $D_i^{reg}$. Since $D_i$ and $Y$ are compatibly split, the splitting of $X$ restricts to a splitting of $D_i^{reg}$, compatibly splitting $Y_i$. Since ${\rm dim}(D_i^{reg})<n$, by induction hypothesis, we conclude that $Y_i$ is invariant under $V$.
	
	The ideal sheaf $\mathcal{I}_Y$ is the product of the $\mathcal{I}_{Y_i}$'s. Let $f_i$ be a section of $\mathcal{I}_{Y_i}$. By the product rule one has $[V,f_1 \ldots f_m]= \sum_{i=1}^m f_1 \ldots \widehat{f}_i \ldots f_m [V,f_i]$. Therefore $Y$ is invariant under $V$.
	
	To prove the second part of the theorem, suppose the splitting is defined by a Poisson $\TT$-Pfaffian $\sigma$ and let $Y$ be a compatibly-split subvariety. Then $[V,\sigma]=0$ by Proposition \ref{hamiltonvanish}, where $V$ is any Hamiltonian vector field or generating vector field of the $\TT$-action. Hence, by the first part of the theorem, $Y$ is invariant under all Hamiltonian vector fields or generating vector fields of the $\TT$-action. But this just means $Y$ is a $\TT$-Poisson subvariety.
\end{proof}

\sectionnew{A Frobenius Splitting of $G/U$ via Poisson $T$-Pfaffian}
\subsection{Frobenius Splitting of Bott-Samelson varieties via Poisson $T$-Pfaffians}\label{bottdef}
Let $G$ be a simply-connected semi-simple algebraic group over an algebraically closed field $k$ of arbitrary characteristic.
Let $\bfu = (s_{\alpha_1}, \ldots, s_{\alpha_n}) = (s_1, \ldots, s_n)$ be any word in $W = N_\sG(T) / T$, the Weyl group of $G$. Recall that the Bott-Samelson variety $Z_\bfu$ associated to the word $\bfu$ is defined to be the quotient space of $P_\bfu = P_{\alpha_1} \times \cdots \times P_{\alpha_n} \subset G^n$ by the $B^n$ action on $P_\bfu$ given by
\[
(p_1, \ldots , p_n) \cdot (b_1, \ldots, b_n)=(p_1b_1,\, b_1^{-1}p_2b_2, \,\ldots, \, b_{n-1}^{-1}p_nb_n),
\]
where $p_i \in P_{\alpha_i}$ and $b_i \in B$ for all $1 \leq i \leq n$. It is well known that $Z_{\bfu}$ is an
$n$-dimensional non-singular projective variety \cite[Section 2.2]{BK}.

Recall that a {\it sub-expression} of a word $\bfu=(s_1,s_2,\ldots,s_n)$ is an element $\gamma \in \{e, s_1 \} \times \{e, s_2 \} \times \ldots \times \{e, s_n \}$, where $e$ is the identity element in $W$.  The set of all sub-expressions of $\bfu$ is denoted by $\Upsilon_{\bfu}$. Each $\gamma=(\gamma_1, \ldots, \gamma_n) \in \Upsilon_{\bfu}$ determines an open embedding to $Z_{\bf u}$ via
\begin{equation}\label{eq-phi-gamma}
\Phi^{\gamma}: \;\; k^n \rightarrow Z_\textbf{u}, \;\; (z_1,\ldots,z_n) \mapsto [u_{-\gamma_1(\alpha_1)}(z_1)\dot{\gamma_1},\ldots,u_{-\gamma_n(\alpha_n)}(z_n)\dot{\gamma_n}].
\end{equation}
The image of $\Phi^{\gamma}$ will be denoted by $\calO^{\gamma}$. The atlas $\{ (\calO^{\gamma}, \Phi^{\gamma}): \gamma \in \Upsilon_\bfu \}$ will be called the standard atlas on $Z_{\bfu}$.

There is a natural $T$-action on $Z_{\bfu}$ given by
\begin{equation}\label{TBS}
t\cdot[p_1,p_2,...,p_n]=[tp_1,p_2,...,p_n],\ \ t\in T,\ p_i\in P_{\alpha_i},\ 1\leq i\leq n.
\end{equation}
Moreover, by direct computation we see that $\calO^{\gamma}$ is $T$-invariant for each $\gamma \in \Upsilon_{\bfu}$ with \[
t \cdot \Phi^{\gamma}(z_1,\ldots,z_n)=\Phi^{\gamma}(t^{-\gamma^1(\alpha_1)}z_1,\ldots,t^{-\gamma^n(\alpha_n)}z_n),
\]
where $\gamma^i :=\gamma_1\gamma_2\ldots \gamma_i \in W$ for $1 \leq i \leq n$.

For the case $k = \CC$, recall from \cite[Section 2.2]{Bott-Samelson} that the product Poisson structure $\pi_{\rm st}^n$ on $P_{\mathbf{u}}$ projects to a well-defined Poisson structure $\pi_{\bf u}$ on $Z_{\bf u}$, which is invariant under the $T$-action.  Therefore, $(Z_{\bfu},\pi_{\bfu})$ is a $T$-Poisson variety. The following theorem is proved in \cite[Theorem 4.14]{Bott-Samelson} \cite[Theorem 5.12]{Bott-Samelson} for the case $k = \CC$.

\begin{theorem}\label{A1}
	For any $\gamma = (\gamma_1, \ldots, \gamma_n) \in \Upsilon_u$,
	in the coordinates $(z_1,\ldots z_n)$ on
	the affine chart $\calO^{\gamma}$ of $Z_{\bf u}$, the Poisson structure $\pi_{\bf u}$  is given by
	\[
	\{z_i,\, z_k \} = \begin{cases}
	\langle \gamma^i(\alpha_i), \, \gamma^k(\alpha_k) \rangle z_iz_k,  &\mbox{if }\gamma_i=e \\
	-\langle \gamma^i(\alpha_i), \, \gamma^k(\alpha_k) \rangle z_iz_k -\langle \alpha_i, \alpha_i \rangle \delta_i(z_k) &\mbox{if }\gamma_i=s_i
	\end{cases},
	\quad 1\leq i < k \leq n,
	\]
	where $\delta_i$ is some vector field on $Z_{(s_{i+1}, \ldots, s_n)}$, and
	$(z_{i+1}, \ldots, z_n)$ are regarded as coordinates on $\calO^{(\gamma_{i+1}, \ldots, \gamma_n)}
	\subset Z_{(s_{i+1}, \ldots, s_n)}$ in the expression $\delta_i(z_k)$ for $k \geq i+1$. Moreover, $(\CC[ \calO^{\gamma}], \pi_{\bf u}|_{\calO^{\gamma}})$ is a Poisson CGL extension for each $\gamma \in \Upsilon_{\bfu}$ with the $h_i$-eigenvalue of $z_i$ given by $\lambda_i(h_i) = \langle \alpha_i, \alpha_i \rangle$ if $\gamma_i = e$ and $\lambda_i(h_i) = - \langle \alpha_i, \alpha_i \rangle$ if $\gamma_i = s_i$.
\end{theorem}

\begin{remark}\label{sdpoissonpcc}
	{\rm	As mentioned in \cite[Section 1.2]{Bott-Samelson}, the formula in Theorem \ref{A1} is algebraic and of integral coefficients.}
\end{remark}

From now on let $G$ be a simply-connected semi-simple algebraic group over an algebraically closed field $k$ of characteristic $p > 3$, equipped with the standard Poisson structure defined in Definition \ref{sdpoissonpc}. The Weyl group of $G$ is naturally isomorphic to the Weyl group of its complex counterpart. We want to point out that the proof of Proposition \ref{qp}, which is carried out on the Lie algebra level, does not depend on the characteristic of the ground field and can be naturally adapted to the non-singular algebraic case. Therefore the product Poisson structure projects to a well-defined Poisson structure on Bott-Samelson varieties and $G/U$ in this positive characteristic case as well. By the same reasoning, Bott-Samelson varieties and $G/U$ over $k$ are $T$-Poisson varieties. See also \cite[Remark 5.21]{Bott-Samelson} for reference.

Moreover, the computation of the local expression of the Poisson bracket in Theorem \ref{A1}, which is again carried out at the Lie algebra level in the Chevalley basis, can naturally be adapted to the positive characteristic case via reduction modulo $p$. Therefore when $G$ is a simply-connected semi-simple algebraic group over an algebraically closed field of characteristic $p > 3$, the Poisson bracket on $Z_\bfu$ in the charts defined by (\ref{eq-phi-gamma}), where $\bfu$ is a word of the Weyl group of $G$, shares the same formula as the Poisson bracket on $Z_\bfu$, where $\bfu$ is regarded as a word of the Weyl group of the complex counterpart of $G$. Since $p > 3$, $\langle \alpha_i, \alpha_i \rangle$ is non-zero modulo $p$. Hence in characteristic $p$, the Poisson structure on the chart $\calO^{\gamma}$ is also a Poisson CGL extension for each $\gamma \in \Upsilon_{\bfu}$.

\begin{theorem}\label{bottsplit}
	A Frobenius Poisson $T$-Pfaffian exists on $(Z_{\bfu},\pi_{\bfu})$ and is unique up to scalar multiple. Moreover, it vanishes on all the sub-Bott-Samelson variety $Z_{\bfu[i]}$.
\end{theorem}

\begin{proof}
	The existence and uniqueness follows from Proposition \ref{A} and Theorem \ref{A1}. In the standard atlas, the Poisson $T$-Pfaffian on the open subset $\calO^{\bf e}$ of $Z_{\bfu}$ is expressed as a scalar multiple of
	\[
	z_1z_2 \ldots z_n \frac{\partial}{\partial z_1} \wedge \ldots \wedge \frac{\partial}{\partial z_n}.
	\]
	The last claim follows from the fact that $z_i$ is the local defining equation for the sub-Bott-Samelson variety $Z_{\bfu[i]}$ in $\calO^{\bf e}$.
\end{proof}
\begin{remark}
{\rm As the Frobenius Poisson $T$-Pfaffian in Theorem \ref{bottsplit} vanishes on all the sub-Bott-Samelson variety $Z_{\bfu[i]}$, it provides a choice of anti-canonical section in \cite[Theorem 2.2.3]{BK}.
}
\end{remark}

\subsection{Proof of Theorem E}
To prepare for the proof of Theorem E, we first prove the following Theorem \ref{GBsplit}.

\begin{theorem}\label{GBsplit}
	A Frobenius Poisson $T$-Pfaffian exists on $(G/B,\pi_{\sG/\sB})$ and is unique up to scalar multiple.
\end{theorem}
\begin{proof}
Recall from \S\ref{T-mixed} that $(G/B,\pi_{\sG/\sB})$ is a $T$-Poisson variety. Let $w_0$ be the longest element of $W$. Given a reduced word $w_0=s_1s_2...s_n$, we get a sequence of simple reflections $\textbf{w}_0=(s_1,s_2,...,s_n)$.
Let $\mu_{\textbf{w}_0}$ be the morphism $Z_{\textbf{w}_0}\rightarrow G/B$ defined by
\begin{equation}\label{resolution}
\mu_{\textbf{w}_0}:Z_{\textbf{w}_0}\rightarrow G/B,\ \ \ [p_1,p_2,...,p_n]\mapsto p_1p_2...p_n.B.
\end{equation}
It is obvious that $\mu_{\textbf{w}_0}$ is $T$-invariant under (\ref{TBS}). Denote by $p_{\textbf{w}_0}$ the projection $P_{\alpha_1} \times \ldots \times P_{\alpha_n} \rightarrow Z_{\textbf{w}_0}$. We know that $p_{\textbf{w}_0}$ is surjective and Poisson.
Since $\mu_{\textbf{w}_0} \circ p_{\textbf{w}_0}$ equals the composition of the multiplication map $P_{\alpha_1} \times \ldots \times P_{\alpha_n} \rightarrow G$ with the projection map $G \rightarrow G/B$, the map
\[\mu_{\textbf{w}_0} \circ p_{\textbf{w}_0}: (P_{\alpha_1} \times \ldots \times P_{\alpha_n},\pi_{\rm st}^n)\rightarrow (G/B,\pi_{\sG/\sB})\]
  is Poisson.  Therefore the morphism
\[\mu_{\textbf{w}_0}: (Z_{\textbf{w}_0},\pi_{\textbf{w}_0})\rightarrow (G/B,\pi_{\sG/\sB})\]
is Poisson. In addition, $\calO^{\bf e} \subset Z_{\textbf{w}_0}$ is $T$-invariant under (\ref{TBS}) and $\mu_{\textbf{w}_0}$ restricts to a $T$-equivariant isomorphism from $\calO^{\bf e}$ to $Bw_0B/B\cap B_-B/B \subset G/B$. It follows from Theorem \ref{bottsplit} that a Frobenius Poisson $T$-Pfaffian exists on $G/B$ and is unique up to scalar multiple.
\end{proof}

For the convenience of the readers, we restate Theorem E.

\begin{theorem}
There exists a unique (up to scalar multiples) Poisson $T$-Pfaffian $\sigma$ on $(G/U,\pi_{\sG/\sU})$ and $\sigma^{p-1}$ induces a Frobenius splitting on $G/U$.
 \end{theorem}
\begin{proof}
	Let $w_0$ be the longest element of $W$. We know from Theorem \ref{GBsplit} that a Frobenius Poisson $T$-Pfaffian exists for $(Bw_0B/B,\pi_{\sG/\sB})$. Suppose it is given by
\begin{equation}\label{sigma}
	\sigma=\pi_{\sG/\sB}^{[r]} \wedge v_{x_1} \wedge \ldots \wedge v_{x_{n-2r}}
\end{equation}
	where $x_1, \ldots, x_n \in \ft$. Now let $\widetilde{x}_i=(x_i,0) \in \ft \oplus \ft$. Suppose the rank of $T$ is $m$. For $i=1, \ldots, m$, let $y_i=(0, e_i) \in \ft \oplus \ft$, where $e_i \in k^m$ is $1$ on the $i$-th spot and $0$ elsewhere. Here we have made the identification between $k^m$ and $\ft$ via $(k^{\times})^m \cong T$. Consider the section
	\[
	\widetilde{\sigma}=\pi_{\sG/\sU}^{[r]} \wedge v_{\widetilde{x}_1} \wedge \ldots \wedge v_{\widetilde{x}_{n-2r}} \wedge v_{y_1} \wedge \ldots \wedge v_{y_m} \in H^0(G/U,  \omega^{-1}_{\sG/\sU}).
	\]
	Since $Bw_0B/U$ is a $T$-invariant open subset of $G/U$, we may regard $\pi_{\sG/\sU}$ as a Poisson structure on $Bw_0B/U$ and $v_x$ as a generating vector field for the $T$-action on $Bw_0B/U$. Then the restriction of $\sigma_1$ on $Bw_0B/U$ is given by the same expression as the corresponding restricted $\pi_{\sG/\sU}$ and $v$.
	
Choose a reduced word for $w_0=s_1s_2...s_n$ and get a sequence of simple reflections $\textbf{w}_0=(s_1,s_2,...,s_n)$.
In the toric coordinates $(\varphi^\textbf{e}_{\textbf{w}_0})^{-1}$ on $(Bw_0B/B\cap B_-B/B)$ defined by
\[
\varphi^\textbf{e}_{\textbf{w}_0}: (k^{\times})^{n}\rightarrow  Bw_0B/B\cap B_-B/B,\ \ \     (z_1,...,z_n)\mapsto (\mu_{\textbf{w}_0}\circ\Phi^{\textbf{e}})(z_1,...,z_n),
\]
where $\Phi^{\textbf{e}}$ is defined in (\ref{eq-phi-gamma}) and  $\mu_{\textbf{w}_0}$ is defined in (\ref{resolution}),
we know from the previous section that $\sigma$ in (\ref{sigma}) is given by a scalar multiple of
	\[
	z_1 \ldots z_n \frac{\partial}{\partial z_1} \wedge \ldots \wedge \frac{\partial}{\partial z_n}.
	\]
With respect to the diagonal action of $T$ on $T\times (Bw_0B/B)$, the map
\[\kappa_{\dot w_0}:Bw_0B/U \rightarrow T\times (Bw_0B/B)\]
 defined by (\ref {eqkappa}) is a $T$-equivariant isomorphism. 
 Then
	\[
	\kappa_{\dot w_0}{\widetilde{\sigma}}=(\kappa_{\dot w_0}\pi_{\sG/\sU})^{[r]} \wedge v_{\widetilde{x}_1} \wedge \ldots \wedge v_{\widetilde{x}_{n-2r}} \wedge v_{y_1} \wedge \ldots \wedge v_{y_m} .
	\]
	Here $v$ represents the generating vector field for the $T $-action on $(Bw_0B/B) \times T$.

Now consider the parametrization
 \[(k^{\times})^{m} \times (k^{\times})^{n} \rightarrow T\times (Bw_0B/B\cap B_-B/B).\]
Denote the coordinates on $T$ by $t_1, \ldots ,t_m$.
	It is obvious that $v_{y_i}=t_i \frac{\partial}{\partial t_i}$. Therefore
	\[
	v_{y_1} \wedge \ldots \wedge v_{y_m}=t_1 \ldots t_m  \frac{\partial}{\partial t_1} \wedge \ldots \wedge \frac{\partial}{\partial t_m}.
	\]
	On the other hand, it follows from Proposition \ref{mixedGU} that the expression of $\kappa_{\dot w_0} \pi_{\sG/\sU}$ must be of the form
	\[
	\pi_{\sG/\sB} + \sum_{i,j} f_{ij}\frac{\partial}{\partial z_i} \wedge \frac{\partial}{\partial t_j}.
	\]
	Since the $\pi_{\sG/\sB}$ terms are the only terms in $\kappa_{\dot w_0}\pi_{\sG/\sU}$ that do not involve $\frac{\partial}{\partial t_j}$, we must have
	\[
	\begin{aligned}
	\kappa_{\dot w_0}{\widetilde{\sigma}}=& \pi_{\sG/\sB}^{[r]} \wedge v_{\widetilde{x}_1} \wedge \ldots \wedge v_{\widetilde{x}_{n-2r}} \wedge v_{y_1} \wedge \ldots \wedge v_{y_m} \\
	=& z_1 \ldots z_n t_1 \ldots t_m  \frac{\partial}{\partial z_1} \wedge \ldots \wedge \frac{\partial}{\partial z_n} \wedge \frac{\partial}{\partial t_1} \wedge \ldots \wedge \frac{\partial}{\partial t_m}.
	\end{aligned}
	\]
	Hence $\kappa_{\dot w_0}{\widetilde{\sigma}}$ gives a Frobenius splitting on $(Bw_0B/B) \times T$, which implies that $\widetilde{\sigma}$ gives a Frobenius splitting on $Bw_0B/U$ and hence on $G/U$.
	
	Proposition \ref{mixedGU} implies that $\kappa_{\dot w_0}\pi_{\sG/\sU}$ is log-canonical in the same parametrization as above. Now $\widetilde{\sigma}$ is a Poisson $T$-Pfaffian and uniqueness also follows from the the local expression of $\kappa_{\dot w_0}\widetilde{\sigma}$.
\end{proof}

\end{document}